\documentclass[11pt,a4paper]{article}

\usepackage{amssymb,stmaryrd, amsmath,amsfonts,
latexsym}
\usepackage{pstricks,pst-plot,pst-node,pst-text,
pst-3d}
\usepackage{enumerate}
\usepackage{geometry,booktabs, fancyhdr}
\usepackage[english]{babel}
\usepackage{comment}
\usepackage[section]{placeins}
\usepackage{amsmath,latexsym,mathrsfs,textcomp, 
fontenc,yhmath}
\usepackage{amssymb,amsfonts,amscd,amssymb,amsthm}
\usepackage{graphics,verbatim,url}
\usepackage{epic, curves}
\newcommand{\doi}[1]{\url{http://dx.doi.org/#1}}
\usepackage{index}

\usepackage{tikz}
\usepackage{caption}
\usepackage{subcaption}

\numberwithin{equation}{section}

\newtheorem{theorem}{Theorem}[section]
\newtheorem{lemma}[theorem]{Lemma}
\newtheorem{corollary}[theorem]{Corollary}

\newcommand{\R}{{\mathbb R}}

\newcommand{\Z}{{\mathbb Z}}
\newcommand{\N}{{\mathbb N}}

\newcommand{\mS}{\mathbb S}

\newcommand{\qall}{\quad\forall}
\newcommand{\cD}{{\mathcal D}}

\newcommand{\cM}{\mathcal M}
\newcommand{\cN}{{\cal N}}

\newcommand{\cO}{\mathcal O}
\newcommand{\cP}{\mathcal P}

\newcommand{\cR}{\mathcal R}

\newcommand{\inp}[2]{\left\langle{#1},{#2}
\right\rangle}
\newcommand{\inpro}[2]{\left\langle{#1},{#2}
\right\rangle}
\newcommand{\inprod}[2]{\left\langle{#1},{#2}
\right\rangle}

\newcommand{\norm}[2]{\left\|{#1}\right\|_{#2}}

\newcommand{\snorm}[2]{\left|{#1}\right|_{#2}}

\newcommand{\brac}[1]{\left(#1\right)}

\newcommand{\sett}[1]{\left\{#1\right\}}

\newcommand{\abs}[1]{\left|#1\right|}

\newcommand{\bsb}[1]{\boldsymbol{#1}}

\newcommand{\De}{\Delta}

\newcommand{\Om}{\Omega}
\newcommand{\al}{\alpha}

\newcommand{\ga}{\gamma}
\newcommand{\om}{\omega}

\newcommand{\vecx}{\boldsymbol{x}}
\newcommand{\vecy}{\boldsymbol{y}}

\newcommand{\veczero}{\boldsymbol{0}}

\newcommand{\vx}{\boldsymbol{x}}
\newcommand{\vy}{\boldsymbol{y}}
\newcommand{\vp}{\boldsymbol{p}}
\newcommand{\vq}{\boldsymbol{q}}

\newcommand{\vv}{\boldsymbol{v}}

\newcommand{\vecv}{\boldsymbol{v}}

\newcommand{\spann}{{\rm{span}}}
\newcommand{\dimm}{{\rm dim}}

\newcommand{\diam}{{\rm{diam}}}

\newcommand{\goto}{\rightarrow}
\newcommand{\wtd}{\widetilde}
\newcommand{\wth}{\widehat}

\newcommand{\slinN}{\sum_{\ell=0}^\infty\,\sum_{
m=-\ell}^{
\ell}}

\newcommand{\Ylm}{Y_{\ell,m}}

\newcommand{\vhlm}{\widehat{v}_{\ell,m}}

\DeclareMathOperator{\cardd}{{card \/}}

\DeclareMathOperator{\supp}{{supp \/}}

\DeclareMathOperator{\dett}{{det \/}}

\DeclareMathOperator{\intt}{{int \/}}
\DeclareMathOperator{\curl}{{curl \/}}

\title{
{
A posteriori error estimates for 
hypersingular integral equation on spheres
with spherical splines}
}

\author{
Duong Pham\thanks{ 
Vietnamese German
University, Le Lai street, Binh Duong New City, 
Binh Duong Province, Vietnam,
\texttt{ptduong01@gmail.com}}
\ 
and
Tung Le\thanks{Vietnamese German
University, Le Lai street, Binh Duong New City, 
Binh Duong Province, Vietnam,
\texttt{le.tung.hcmus@gmail.com}}
}

\date{\today}
\makeatletter \@addtoreset{equation}{section}
\@addtoreset{table}{section} \makeatother
\begin{document}

\maketitle

\begin{abstract}
{
A posteriori residual and hierarchical upper bounds for the 
error estimates were proved when solving the 
hypersingular integral equation on the unit sphere by using the Galerkin method 
with spherical splines.} 
Based on these a posteriori error estimates, 
adaptive mesh refining procedures are used to reduce
complexity and computational cost of the discrete problems. 
Numerical experiments illustrate our theoretical results.

\end{abstract}


\vspace{0.3cm}


\section{Introduction}\label{Int}

Hypersingular integral equations have many 
applications, for example in
acoustics, fluid mechanics, elasticity and 
fracture mechanics
\cite{Kle05}.
These equations arise
from the boundary-integral reformulation of the 
Neumann
problem with the Laplacian in 
a bounded or unbounded domain, see e.g.~\cite{HsiWen08,Steinbach08}.
In this paper, we 
study the hypersingular integral equation on the 
unit sphere
\begin{equation} \label{eqn:Heqn}
 -N{u} + \omega^2\int_{\mS} u\,d\sigma = f \quad 
\text{on } \mS,
\end{equation}
where $N$ is the hypersingular integral operator 
given by
\begin{equation}\label{equ:hyper define}
N v(\vecx) := \frac{1}{4\pi} 
\frac{\partial}{\partial\nu_{\vecx}}
\int_{\mS} v(\vecy) 
\frac{\partial}{\partial\nu_{\vecy}}
\frac{1}{\snorm{\vecx-\vecy}{}} d\sigma_{\vecy},
\end{equation}
$\omega$ is some nonzero real constant,
and
$\mS$ is the unit sphere in $\mathbb{R}^3$, that 
is,
 $\mS = \{{\vecx} \in \mathbb{R}^3 : 
\snorm{\vecx}{} = 1\}$.
Here $\partial/\partial\nu_{\vecx}$ is the normal 
derivative
with respect to $\vecx$, and
$\snorm{\cdot}{}$ denotes the Euclidean norm. 
{The hypersingular integral equation on the unit sphere has applications in 
geophysics where people are solving Neumann problems in the interior or 
exterior of the surface of the Earth, see e.g.~\cite{FreGerSch98, 
GraKruSch03, Ned00, Sve83, 
TraLeGSloSte09a}.
Efficient solutions to the hypersingular integral equation on the sphere become
more demanding when given data are collected by satellites.
}

The equation~\eqref{eqn:Heqn} can be solved by using 
tensor products of univariate splines on regular grids which do not 
exist  when the data is given by satellites. 
Spherical radial basis functions appear to be more suitable for solving 
problems with scattered data, see e.g.~\cite{ 
MorNea02, NarWar02, PhamTran2014, TraLeGSloSte09a} and references therein.
However, the
resulting matrix system from this approximation is very ill-conditioned. 
{Even 
though
overlapping additive Schwarz preconditioners can be designed for this problem, 
the
condition number of the preconditioned system still depends on the number of 
subdomains
and the angles between subspaces; see~\cite{TraLeGSloSte10}.
}

{
The space of spherical splines defined on a spherical
triangulation seems particularly appropriate for use on the 
sphere~\cite{AlfNeaSch96a, AlfNeaSch96b}.
It consists of functions whose
pieces are spherical homogeneous polynomials joined together
with global smoothness, and thus has both the
smoothness and high degree of flexibility~\cite{FasSch98}.
That flexibility makes spherical splines become a powerful tool.
These splines have been used successfully in interpolation and data
approximation on spheres, see~\cite{AlfNeaSch96c, NeaSch04}.
In an attempt to use spherical splines in solving partial differential
equations,
Baramidze and Lai~\cite{BarLai05} use these functions to solve the
Laplace--Beltrami equation on the unit sphere. Later, Pham et al. use spherical 
splines to solve pseudodifferential equations on the unit 
sphere~\cite{PhamTranChernov11}.
The use of spherical splines
has some significant advantages. One of them is the ability 
to write the approximate solutions of the equations in  the form of linear 
combinations of Bernstein--B\'ezier polynomials which
play an extremely important role in computer aided geometric
design, data fitting and interpolation, computer vision and elsewhere; see 
e.g.~\cite{Far88, HosLas93}. Another advantage is
the ability to control the smoothness of a function and its
derivatives across edges of the triangulations; see~\cite{AlfNeaSch96a}.
}

In this paper, the hypersingular integral equation~\eqref{eqn:Heqn} will be 
solved by using 
the Galerkin method with
 spherical 
splines. 
{The linear system arising when solving this equation  by using spherical 
splines 
is also ill-conditioned. However, 
preconditioners can be used to tackle this problem, see~\cite{PhamTran12}.}
When solving the hypersingular integral equation~\eqref{eqn:Heqn} by using the 
Galerkin method with spherical splines associated with a regular and 
quasi-uniform spherical triangulation $\De$, 
an a priori error estimate is proved as follows 
\begin{equation}\label{a priori hyper}
\norm{u - u_{\De}}{H^{1/2}(\mS)}
\le C
h_{\De}^{s-1/2} 
\norm{u}{H^s(\mS)},
\end{equation}
see Theorem~5.1 in~\cite{PhamTranChernov11}.
Here, $s$ is any real number satisfying 
$1/2\le s\le d+1$ where $d$ is the degree of spherical splines, and
$C$ is a constant which is independent of 
the mesh size $h_{\De}$ and the 
exact 
(unknown) 
solution $u$.
The a priori error
estimate~\eqref{a priori hyper} reveals the rate of convergence 
in which the upper bound for the approximation error depends on the mesh size 
$h_{\De}$ and the unknown exact solution.
{However, the quasi-uniform condition on the mesh 
suggests 
that
uniform refinements of all spherical triangles must be applied when one wish to 
improve approximation quality.
This may lead to an unnecessary waste of computational efforts since 
contributions to the total error vary over different regions on the unit sphere.
}

A posteriori error estimates can 
provide numerical estimates of accuracy in terms 
of the source term and discrete solutions. 
In this paper, we shall prove {two kinds of a posteriori 
upper  bounds} for the errors when 
solving the hypersingular integral equation on 
the 
unit sphere by using Galerkin method with 
spherical splines.
{Firstly, we shall prove an a posteriori residual estimate} (see 
Theorem~\ref{t:upp bound}),
\begin{align}
\norm{u-u_{\De}}{H^s(\mS)}
\le 
C
\brac{ 
\sum_{\tau\in \De}
h_{\tau}^{2-2s}
\norm{f+ N u_{\De} - \om^2\inpro{u_{\De}}{1}}{L_2(\tau)}^2
}^{1/2},
\label{resi poster est}
\end{align}
where $s\in[0,1/2]$ and $C$ is a positive constant depending only on the 
smallest angle of $\De$. 
Here, the approximate solution $u_{\De}$ is found in the space 
$S_d^r(\De)$ of spherical splines of order $d$ and smoothness $r$ 
associated with $\De$
where $\De$ is a regular spherical triangulation.
Secondly, 
when the approximate solution $u_{\De}$ is found in the space of continuous 
{piecewise} linear spherical splines,
we shall prove another  
a posteriori error estimate (the hierarchical estimate),  
\begin{align}
\norm{u-u_{\De}}{H^{1/2}(\mS)}^2 
\le 
C
\sum_{\tau\in \De}
\sum_{
\vv_i\in V_{\De'}
\atop 
\vv_i\in \tau
}
\brac{
\frac{\inpro{f+N u_{\De} - \om^2 \inpro{u_{\De}}{1}} 
{B_{\vv_i}'}
}{\norm{B_{\vv_i}'}{H^{1/2}(\mS)}}
}^2,
\label{hier est tau intro}
\end{align}
see Corollary~\ref{c:Hier est}. Here, $\De'$ is a {fictional} refinement 
of 
$\De$ so that a saturation assumption is satisfied, $V_{\De'}$ is the 
set of all vertices of $\De'$, and $B_{\vv_i}'$ are nodal basis functions 
 associated with vertices $\vv_i$ of $\De'$. Precise definitions of spherical 
triangulations, spherical splines and their basis functions, and Sobolev spaces 
defined on the unit sphere $\mS$ will be presented in Section~\ref{s:Pre}.

Based on {these a posteriori error 
estimates, \eqref{resi poster est} and~\eqref{hier est tau intro},} we 
use adaptive mesh 
refinement techniques to create better approximation spaces. 
This results in a significant reduction in 
required degrees of freedom and computation 
time while preserving approximate accuracy.
{
This improvement is very important when we are solving geophysical 
problems which require considerably large numbers of data points. Furthermore, 
although all the results in this paper are established for problems on the unit 
sphere, they can be extended to more general (but related to the sphere) 
geometries, such as sphere-like geometries (see e.g.~\cite{AlfNeaSch96c, 
ChernovPham14, HuaYu06, LeGSteTra11}). This possible extension can broaden 
applications of our research.
}

The structure of the paper is as follows. In 
Section~\ref{s:Pre}, we will review
spherical splines, introduce the Sobolev spaces 
on 
the unit sphere to be used,
present the quasi-interpolation operator and 
the hypersingular integral equation. 
The proof for an a posteriori residual upper bound for the 
error estimate  is presented in 
Section~\ref{s:Error}. 
{In Section~\ref{s:hierarchical sec}, hierarchical basis techniques are 
used to 
prove a posteriori hierarchical error estimate when 
solving~\eqref{eqn:Heqn} 
by using continuous {piecewise} linear spherical splines.}
In 
Section~\ref{s:mesh ref}, we discuss  simple 
adaptive mesh refinement algorithms based on the 
a posteriori error estimates.
 The final section
(Section~\ref{s:num ex}) presents our numerical experiments 
which illustrate our theoretical results.

In this paper $C$ and $C_i$, for $i=1,\ldots, 5$, 
 denote generic constants which may take
different values at {different occurrences.}

\section{Preliminaries}\label{s:Pre}
In this section, we will first review spherical 
splines 
\cite{AlfNeaSch96a, AlfNeaSch96b, AlfNeaSch96c} 
and introduce our functional spaces on the unit 
sphere 
$\mS \subset \mathbb{R}^3$. Then the 
quasi--interpolation operator and the 
hypersingular integral equation will be discussed.

\subsection{Spherical splines}

The {\it trihedron} 
$T$ generated by three linearly independent vectors $\{\vecv_1, \vecv_2, 
\vecv_3\}$  in $\R^3$
is defined by
\[
T
=
\{
\vecv\in \R^3: \vecv = b_1\vecv_1 + b_2 \vecv_2 + 
b_3 \vecv_3\ {\rm with}\ b_i\ge 0,
\ i = 1,2,3
\}.
\]
The intersection $\tau= T\cap \mS$ is called a 
{\it spherical triangle}.
Let $\De = \{\tau_i: i = 1,\ldots, \mathcal{T}\}$ 
be a set of spherical 
triangles. 
Then $\De$ is called a {\it spherical 
triangulation} of the sphere $\mS$
if there hold
\begin{itemize}
\item[(i)] 
$\bigcup_{i=1}^{\mathcal T} \tau_i = \mS$,
\item[(ii)] 
each pair of {distinct} triangles in $\De$ are 
either disjoint 
or share a common vertex or an edge.
\end{itemize} 
% Figure~\ref{f:uniform4098} shows an example of spherical triangulation that is 
% a uniform spherical triangulation with 4098 vertices.
Let $\Pi_d$ denote the space of 
trivariate 
homogeneous polynomials of degree $d$ in $\R^3$. 
The space of restrictions on the unit sphere 
$\mS$ of all polynomials in $\Pi_d$ is denoted by 
$\Pi_d(\mS)$. Similarly, we also denote by 
$\cP_d$ and $\cP_d(\mS)$ the spaces of 
polynomials of degree $d$
in $\R^3$ and on $\mS$, respectively.
We define $S_d^r(\De)$ to be the space of 
{piecewise} homogeneous splines 
of degree $d$ and smoothness~$r$ on a spherical 
triangulation $\Delta$, that is,
\begin{equation*}
S_d^r(\De) = \{ s \in C^r(\mS) : s|_{\tau} \in 
\Pi_d, \tau \in \Delta\}.
\end{equation*}
Throughout this paper, we always assume that 
\begin{equation}\label{dr cond}
\begin{cases}
d \ge 3r + 2
&
\text{if } r\ge 1
\\
d \ge 1 
&
\text{if } r = 0
\end{cases}
\end{equation}
holds;
see \cite{AlfNeaSch96a, AlfNeaSch96b, 
AlfNeaSch96c}. 

For a spherical triangle $\tau$ with vertices 
${\vecv}_1, {\vecv}_2$, and ${\vecv}_3$, let 
$b_{1,\tau}({\vecv}), b_{2,\tau}({\vecv})$, and 
$b_{3,\tau}({\vecv})$ denote the spherical 
barycentric coordinates as functions of ${\vecv}$ 
in $\tau$, i.e.,
\begin{equation}\label{b1tau def}
\vecv 
=
b_{1,\tau}({\vecv})\vecv_1 +
b_{2,\tau}({\vecv})\vecv_2
+
b_{3,\tau}({\vecv})\vecv_3.
\end{equation}
{
Suppose that $\vv_i = (v_i^x, v_i^y, v_i^z)$ for $i=1,2,3$ and $\vv = (v^x, 
v^y, v^z)$. Equation~\eqref{b1tau def} defining the coordinates $b_{i,\tau}$, 
for 
$i=1,2,3$, can be written as a system of three linear equations
\begin{equation*}\label{b123 mat eqn}
\begin{pmatrix}
v_1^x & v_2^x & v_3^x
\\
v_1^y & v_2^y & v_3^y
\\
v_1^z & v_2^z & v_3^z
\end{pmatrix}
\begin{pmatrix}
b_{1,\tau}
\\
b_{2,\tau}
\\
b_{3,\tau}
\end{pmatrix}
=
\begin{pmatrix}
v^x\\ v^y \\ v^z
\end{pmatrix}.
\end{equation*}
Using Cramer's rule, we have 
\begin{align}
b_{1,\tau}(\vv) 
=
\frac{\dett (\vv, \vv_2, \vv_3)}{\dett(\vv_1,\vv_2,\vv_3)},
\quad 
b_{2,\tau}(\vv) 
=
\frac{\dett (\vv_1, \vv, \vv_3)}{\dett(\vv_1,\vv_2,\vv_3)},
\quad 
b_{3,\tau}(\vv) 
=
\frac{\dett (\vv_1, \vv_2, \vv)}{\dett(\vv_1,\vv_2,\vv_3)},
\label{b1v deter}
\end{align}
where 
\[
\dett(\vv_1,\vv_2,\vv_3)
:=
\dett
\begin{pmatrix}
v_1^x & v_2^x & v_3^x
\\
v_1^y & v_2^y & v_3^y
\\
v_1^z & v_2^z & v_3^z
\end{pmatrix}.
\]
}

We define the homogeneous 
Bernstein basis polynomials of degree $d$ 
relative to 
$\tau$ to be the polynomials
\begin{equation}\label{Berns1}
B_{ijk}^{d,\tau}({\vecv}) = 
\frac{d!}{i!j!k!}b_{1,\tau}({\vecv})^ib_{2,\tau}({
\vecv})^jb_{3,\tau}({\vecv})^k, \quad i + j + k = 
d.
\end{equation}
As was shown in \cite{AlfNeaSch96a}, we can use 
these polynomials
as a basis for $\Pi_d$.

A spherical cap centred at ${\vecx} \in \mS$ and 
having radius $R$ is defined by
\begin{equation}\label{Radius}
 C({\vecx}, R) = \{{\vecy} \in \mS : 
\cos^{-1}({\vecx} \cdot {\vecy}) \leq R\}.
\end{equation}
\noindent For any spherical triangle $\tau$, let 
$|\tau|$ denote the diameter of the smallest 
spherical cap containing $\tau$, and 
$\rho_{\tau}$ 
denote the diameter of the largest spherical cap 
contained in $\tau$. 
We define
\begin{equation*}
|\De| = \max\{|\tau|: \tau\in\Delta\} \quad 
\text{and} \quad \rho_{\De} = \min\{\rho_{\tau}: 
\tau \in \Delta\},
\end{equation*}
and refer to $|\De|$ as the mesh size. 
Our triangulations 
are said to be {\it regular} if for some given 
$\beta >1$, there holds
\begin{equation}\label{regular}
\snorm{\tau}{}
\le 
\beta
\rho_\tau
\quad
\forall
\tau\in\De
\end{equation}
and {\it quasi-uniform} if
for some given positive number  $\gamma < 1$, 
there 
holds 
\begin{equation}\label{quasi uniform}
\snorm{\tau}{}
\ge 
\gamma
\snorm{\De}{×}
\quad
\forall
\tau\in\De.
\end{equation}
Roughly speaking, the regularity guarantees the 
smallest angles in 
our triangulations
are sufficiently large so that there are no too 
narrow triangles
and the 
quasi-uniformity guarantees that the sizes of 
triangles in a triangulation 
are not too much different. 
% Triangulations satisfying both~\eqref{regular} 
% and~\eqref{quasi uniform} are said to be 
% \textit{$\beta$-quasi-uniform}.
    
To accompany the results used in~\cite{BarLai05, 
NeaSch04, PhamTranChernov11} we also denote 
\begin{equation}\label{htau tan}
h_\tau 
=
\tan\brac{\abs{\tau}/2}.
\end{equation}
It is obvious that 
\begin{equation}\label{ht rhotau}
\rho_\tau \le 
\abs{\tau}\le 2 h_\tau
\qall 
\tau \in \De.
\end{equation}
Noting~\eqref{regular} and~\eqref{htau tan}, the 
regularity of a set 
of 
triangulations can also be written by 
\begin{equation}\label{regular2}
h_\tau 
\le 
\beta_1
\tan\brac{\frac{\rho_\tau}{2}}
\quad 
\text{or}
\quad 
h_\tau 
\le 
\beta_2\, 
\rho_\tau
\qall 
\tau\in \De
\end{equation}
for some positive numbers $\beta_1$ and $\beta_2$.
For any $\tau\in\De$, we denote by $A_{\tau}$ the area 
of $\tau$. If $\De$ is regular, there holds 
\begin{equation}\label{Atau htau regular}
\beta_3  
h_{\tau}
\le 
A_{\tau}^{1/2}
\le 
\beta_4 
h_{\tau} 
\quad 
\forall 
\tau \in \De,
\end{equation}
for some positive constants
$\beta_3$ and $\beta_4$. 
Similarly,
the quasi-uniformity can be written as
\begin{equation}\label{quasi2}
h_\tau 
\ge 
\ga_1\,
\abs{\De}
\qall 
\tau\in \De.
\end{equation}
For any $\tau\in \De$, we denote $\Om_{\tau}$ to be 
the union of all triangles in $\De$ which share 
with $\tau$ at least a common vertex or a common 
edge.
If the triangulations $\De$ are regular, there 
holds 
\begin{equation}\label{regu local uniform}
h_\tau 
\ge \beta_5 
\abs{\Om_{\tau}}
\qall 
\tau 
\in \De,
\end{equation}
 for some $\beta_5>0$,
see~\cite[Lemma~4.14]{LaiSch07}.
We denote by $h_{\De}$ the mesh size of $\De$, i.e.,
\begin{equation}\label{hX def}
h_{\De} = 
\tan(\abs{\De}/2).
% \quad 
% \text{where}
% \quad 
% \abs{\De}
% =
% \max\sett{ 
% \abs{\tau} : \tau\in \De
% }.
\end{equation}

We denote by $V_{\De}$ the set of all vertices of the spherical triangulation 
$\De$.
Let $\vecv_i\in V_{\De}$. We also denote by $T_{\vecv_i}^{\De}$ 
the set of 
triangles in $\De$ whose one of their vertices is $\vecv_i$.
If $\De$ is regular, the smallest angle in $\De$ is bounded below. This 
suggests that the numbers of spherical triangles which share a common vertex is 
bounded, i.e., there is a positive integer {$L$} (depending only on the 
smallest 
angle of $\De$) such that 
\begin{equation}\label{Tvi bounded}
\textrm{card}\brac{T_{\vecv_i}^{\De}}
\le {L}
\quad 
\forall 
\vecv_i\in V_{\De}.
\end{equation}

\subsection{Sobolev spaces}

For every $s\in \R$, the Sobolev space $H^s(\mS)$ defined on the whole unit 
sphere $\mS$ can be defined by using 
Fourier expansion with spherical harmonics.
A spherical harmonic of order $\ell$ on $\mS$ 
is the restriction to $\mS$ of a 
homogeneous harmonic
polynomial of degree $\ell$ in $\R^3$.
The space of all spherical harmonics of order $\ell$ 
is the eigenspace of the Laplace--Beltrami operator
$\De_{\mathbb S}$ corresponding to the eigenvalue 
$\lambda_\ell = - \ell(\ell + 1)$. The dimension of this space 
being $2\ell+1$,
see e.g.~\cite{Mul66}, 
one may choose for it an orthonormal basis
$\{Y_{\ell,m}\}_{m=-\ell}^\ell$. 
The collection of all the spherical harmonics
$\Ylm$, $m = -\ell,\ldots, \ell$ and $\ell=0,1,\ldots$, forms an orthonormal 
basis for $L_2(\mS)$.
The Sobolev space $H^s(\mS)$ is defined as usual by
\[
H^s(\mS)
:=
\Big{\{}
v\in\cD'(\mS): \slinN (\ell+1)^{2s} |\vhlm|^2 < \infty
\Big{\}},
\]
where $\cD'(\mS)$ is the space of distributions on
$\mS$ and $\vhlm$ are the Fourier coefficients of~$v$,
\begin{align}
\vhlm
=
\int_{\mS}
v(\vecx)\Ylm(\vecx)\; d\sigma_{\vecx}.
\label{Fourier coeff}
\end{align} 
% Here $d\sigma_{\vecx}$ is the element of surface
% area.
The space $H^s(\mS)$ is equipped with the following norm
and inner product: 
\begin{equation}\label{equ:Sob nor}
  \|v\|_{H^s(\mS)} := 
       \left(\sum_{\ell=0}^\infty \sum_{m=-\ell}^\ell
          (\ell+1)^{2s} |\widehat{v}_{\ell,m}|^2 \right)^{1/2}
\end{equation}
and
\[
\inpro{v}{w}_{H^s(\mS)}
:=
\sum_{\ell=0}^\infty \sum_{m=-\ell}^\ell
(\ell+1)^{2s} \widehat{v}_{\ell,m} {\widehat{w}_{\ell,m}}.
\]
When $s=0$ we write $\inpro{\cdot}{\cdot}$ instead of
$\inpro{\cdot}{\cdot}_{H^0(\mS)}$; this is in fact the $L_2$-inner
product.
We note that
\begin{equation}\label{equ:CS s}
|\inpro{v}{w}_{H^s(\mS)}|
\le
\norm{v}{H^s(\mS)} \norm{w}{H^s(\mS)}
\quad\forall v,w \in H^s(\mS),\ \forall s \in\R,
\end{equation}
and
\begin{equation}\label{equ:neg nor}
\norm{v}{H^{s_1}(\mS)}
=
\sup_{w \in H^{s_2}(\mS) \atop w \not= 0}
\frac{\inpro{v}{w}_{H^{\frac{s_1+s_2}{2}}(\mS)}}{\norm{w}{H^{s_2}(\mS)}}
\quad\forall v \in H^{s_1}(\mS),\ \forall s_1, s_2 \in\R.
\end{equation}
In particular, there holds
\begin{equation}\label{dual1}
\norm{v}{H^{-s}(\mS)}
=
\sup_{
w\in H^{s}(\mS)\atop 
w\not= 0}
\frac{\inpro{v}{w}}{
\norm{w}{H^{s}(\mS)}
}.
\end{equation}

In the case $k$ belongs to the set of nonnegative integers $\Z^+$, the Sobolev
space $H^k(\Omega)$ on a subset $\Omega\subset \mS$ can be defined by
using {an atlas} for the unit sphere $\mS$ \cite{NeaSch04}.
Let $\{(\Gamma_j, \phi_j)\}_{j = 1}^J$ be an atlas for $\Omega$, i.e,
a finite collection of charts $(\Gamma_j, \phi_j)$, where $\Gamma_j$ are
open subsets of $\Omega$, covering $\Omega$, and
where $\phi_j : \Gamma_j \rightarrow B_j$ are infinitely
differentiable mappings
 whose inverses $\phi_j^{-1}$ are also infinitely differentiable.
Here $B_j$, $j = 1,\ldots, J$, are open subsets in $\R^2$.
Also, let $\{\psi_j\}_{j=1}^J$ be a partition \label{pag:par} of unity 
subordinate to the atlas
$\{(\Gamma_j, \phi_j)\}_{j=1}^J$, i.e., a set of infinitely differentiable
functions $\al_j$ on $\Omega$ vanishing outside the sets $\Gamma_j$, such that
$\sum_{j=1}^J \psi_j = 1$ on $\Omega$.
For any $k\in\Z^+$,
 the Sobolev space $H^k(\Omega)$ on the unit sphere is defined as follows
\begin{equation}\label{equ:Sob spa}
H^k(\Omega)
:=
\{v: (\psi_j v)\circ \phi_j^{-1}\in H^k(B_j),\ j=1,\ldots, J \},
\end{equation}
which is equipped with a norm defined by
\begin{equation}\label{equ:nor 2}
\norm{v}{H^k(\Omega)}^*
:=
\sum_{j = 1}^J
\norm{(\psi_j v)\circ \phi_j^{-1}}{H^k(B_j)}.
\end{equation}
Here, $\norm{\cdot}{H^k(B_j)}$ denotes the usual $H^k$-Sobolev norm defined on 
the subset $B_j$ of the plane $\R^2$.
In the case $\Omega = \mS$,
this norm is equivalent to the norm defined in \eqref{equ:Sob nor};
see~\cite{LioMag}.

To accompany the results used in~\cite{BarLai05, 
NeaSch04, PhamTranChernov11}, we also present here a definition of 
Sobolev spaces defined on a subset of $\mS$ by using homogeneous extensions of 
a function defined on $\mS$.
Let $\ell\in \N$ and let $v$ be a function defined on the unit sphere $\mS$. We 
denote by $v_\ell$ the homogeneous extension of degree $\ell$ of $v$ to $\R^3$, 
i.e.,
\begin{align}
v_\ell(\vecx) 
:=
\abs{\vecx}^\ell
v
\brac{ 
\frac{\vecx}{\abs{\vecx}}
},
\quad 
\vecx\in \R^3{\setminus}\sett{\veczero}.
\notag
\end{align}
For every $v\in H^k(\Omega)$, 
we define Sobolev--type seminorms of $v$ by
\begin{equation}\label{equ:sno}
\snorm{v}{H^\ell(\Omega)}
:=
\sum_{|\al|=\ell}
\norm{D^\al v_{\ell-1}}{L_2(\Omega)},
\quad 
\ell=1,\ldots, k.
\end{equation}
Here
$\norm{D^\al v_{\ell-1}}{L_2(\Omega)}$ is understood as the $L_2$-norm
of the restriction of the trivariate function $D^\al v_{\ell-1}$ to $\Omega$.
When $\ell=0$ we define
\[
\snorm{v}{H^0(\Omega)}
:=
\norm{v}{L_2(\Omega)},
\]
which can now be used together with \eqref{equ:sno} to define
a norm in $H^k(\Omega)$:
\begin{equation}\label{equ:nor 3}
\norm{v}{H^k(\Omega)}'
:=
\sum_{\ell=0}^k \snorm{v}{H^\ell(\Omega)}.
\end{equation}
 This norm is equivalent to the norm $\norm{\cdot}{H^k(\Omega)}^*$ defined
 by~\eqref{equ:nor 2}; see \cite{NeaSch04}.

For every $s\in [0,1]$, 
the spaces $\wtd H^s(\Om)$ and $H^s(\Omega)$ 
are defined by Hilbert 
space interpolation~\cite{BerLof} so that
\begin{equation}\label{equ:H 1 over 2}
\wtd H^s(\Omega)
:=
[L_2(\Omega), H_0^1(\Omega)]_s,
\quad 
\text{and}
\quad
H^s(\Omega)
:=
[L_2(\Omega), H^1(\Omega)]_s,
\end{equation}
where 
{$H_0^1(\Om)=\sett{v\in H^1(\Om): v = 0\ \text{on } \partial\Om}$,} and 
$[X_0, X_1]_s$ denotes the  $L_2$-interpolation 
of $X_0$ and $X_1$, see e.g.~\cite{BerLof, McL00}.
Here, $H_0^1(\mS)$ is
the space of all functions in $H^1(\mS)$ which vanish on the boundary $\partial 
\Om$ of $\Om$, i.e.,
\[
H_0^1(\Omega) = \{v\in H^1(\Omega): v = 0\ 
\text{on}\
\partial\Omega\}.
\]
The spaces $H^{-s}(\Omega)$ and $\wtd 
H^{-s}(\Omega)$ are
defined as the dual spaces of $\wtd 
H^{s}(\Omega)$ and
$H^{s}(\Omega)$, respectively, with respect to 
the duality pairing
which is the usual extension of the $L_2$-inner 
product on $\Omega$.
In particular,
the space $H^{-s}(\mS)$ is defined to be the dual 
space of $H^{s}(\mS)$. 
The $\norm{\cdot}{H^s(\mS)}$-norm defined by~\eqref{equ:Sob nor} turns out to 
be equivalent 
to the $\norm{\cdot}{H^s(\mS)}'$-norm defined by~\eqref{equ:nor 
2},~\eqref{equ:H 
1 over 2} 
and~\eqref{dual1} when $\Om = \mS$ and $-1\le s\le1$, i.e.,
\begin{align}
\ga_2 \norm{v}{H^s(\mS)}
\le 
\norm{v}{H^s(\mS)}'
\le 
\ga_3
\norm{v}{H^s(\mS)}
\quad 
\forall v\in H^s(\mS),
\label{norms euiv}
\end{align}
for some positive numbers $\ga_2$ and $\ga_3$, {see e.g.~\cite{Heu14, 
LioMag, 
NeaSch04, Ned00}.}

\subsection{Quasi-Interpolation}\label{s:Q def}

We now briefly discuss the construction of a 
quasi-interpolation operator $Q: 
L_2(\mS)\goto S_d^r(\De)$ which is defined in 
\cite{NeaSch04}. Firstly, we introduce the set of 
\textit{domain points} of $\De$ to be
\[
\cD 
=
\bigcup_{\tau=\langle\vv_1,\vv_2,
\vv_3\rangle\in\De}
\sett{\xi_{ijk}^\tau=\frac{i\vv_1+j\vv_2+k\vv_3}{d
}}_{i+j+k=d}
\]
Here, $\tau=\langle\vv_1,\vv_2,\vv_3\rangle$ 
denotes the spherical triangle whose vertices are 
$\vv_1,\vv_2,\vv_3$. We denote the domain points 
by $\xi_1, \ldots, \xi_D$, where $D 
= {\dimm}\, S_d^0(\De)$. 
Let $\sett{B_\ell:\ell=1,\ldots,D}$ be 
a basis for $S_d^0(\De)$ such that the 
restriction 
of $B_\ell$ on the triangle containing 
$\xi_\ell$ is $Bernstein$ %-$Br\acute{e}zier$ 
polynomial of degree $d$ associated with this 
point, and that $B_\ell$ vanishes on other 
triangles.

A set $\cM=\sett{\zeta_\ell}_{\ell=1}^{M}\subset 
\cD$ is called a \textit{minimal determining} set 
for $S_d^r(\De)$ if, for every $s\in 
S_d^r(\De)$, 
all the coefficients $\nu_\ell(s)$ in the 
expression $s=\sum_{\ell=1}^D \nu_\ell(s)B_\ell$ 
are uniquely determined by the coefficients 
corresponding to the basis functions which are 
associated with points in $\cM$. Given a minimal 
determining set, we construct a basis 
$\sett{B^*_\ell}_{\ell=1}^M$ for $S_d^r(\De)$ 
by 
requiring
\begin{align}\label{nu def}
\nu_{\ell'(B_\ell^*)}=\delta_{\ell,\ell'}, \quad 
	1\leq \ell,\ell'\leq M.
\end{align}
By using Hahn-Banach theorem we extend the linear 
functions $\nu_\ell$, $\ell=1,\ldots,M$, to all 
of $L_2(\mS)$. We continue to use the same symbol 
for the extensions.
The quasi-interpolation operator: $Q: L_2(\mS) 
\goto S_d^r(\De)$ is now defined by
\begin{equation}\label{Qv def}
Qv=\sum_{\ell=1}^M \nu_\ell(v) B_\ell^*, \quad 
v\in L_2(\mS).
\end{equation}

\subsection{The hypersingular integral equation}

The hypersingular integral operator~\eqref{equ:hyper define} arises from the 
boundary-integral reformulation of the Neumann problem with the Laplacian in 
the interior or the exterior of the sphere, see~\cite{Sve83}. This operator 
(with minus sign) turns out to be a strongly elliptic pseudodifferential 
operator of order $1$, see e.g.~\cite{Sve83, PhamTranChernov11}, i.e.
\begin{align}
-N v(\vecx)
&
=
\slinN
\frac{\ell(\ell+1)}{2\ell+1}
\vhlm 
\Ylm(\vecx),
\quad 
\vecx\in \mS.
\label{hyp def series}
\end{align}
In this paper, we solve
the hypersingular integral 
equation~\eqref{eqn:Heqn},
\begin{align}
-Nu + \om^2 \int_{\mS} u\,d\sigma 
=
f
\quad 
\text{on}
\quad 
\mS,
\label{N u om f S}
\end{align}
where $f\in H^{-1/2}(\mS)$. 
We denote by $\cN^*: H^s(\mS) \goto H^{s-1}(\mS)$ the operator which is given by
\begin{align}
\cN^* v = -Nv + \om^2 \int_{\mS} 
v\,d\sigma,
\quad 
v\in H^s(\mS).
\label{N star def}
\end{align}
Noting~\eqref{equ:Sob nor}, \eqref{hyp def series} and~\eqref{N star def}, we 
have 
\begin{align}
\norm{\cN^* v}{H^{s-1}(\mS)}^2
&
=
\sum_{\ell=1}^\infty 
\sum_{m=-\ell}^\ell 
(\ell+1)^{2s} 
\frac{\ell^2}{(2\ell+1)^2}
\abs{\vhlm}^2 
+
4\pi \om^4 
\abs{\wth v_{0,0}}^2.
\label{N star norm sm1}
\end{align}
For every $\ell\ge 1$, there holds
\[
\frac{1}{9}
\le 
\frac{\ell^2}{(2\ell+1)^2}
\le 
\frac{1}{4}.
\]
This together with~\eqref{equ:Sob nor} and~\eqref{N star norm sm1} implies 
\begin{equation}\label{Nv norm v sm1}
\al_1
\norm{v}{H^s(\mS)}
\le 
\norm{\cN^* v}{H^{s-1}(\mS)}
\le 
\al_2
\norm{v}{H^s(\mS)}
\quad 
\forall 
v\in H^s(\mS),
\end{equation}
where 
\begin{align}
\al_1 
&
=
\min
\sett{\frac{1}{3}, 2\om^2\sqrt{\pi}}
\quad 
\text{and}
\quad 
\al_2 
=
\max
\sett{\frac{1}{2}, 2\om^2\sqrt{\pi}}.
\label{al1 2 def}
\end{align}
To set up a weak formulation, we introduce the 
bilinear form
\begin{equation}\label{auv def}
a(u, v) 
:= 
\inprod{\cN^* u}{v},
% - \inprod{N u}{v} + \omega^2
% \inprod{u}{1}\inprod{v}{1}, 
\quad u,v\in 
H^{1/2}(\mS),
\end{equation}
where 
$\inpro{v}{w}$ 
is the $H^{1/2}(\mS)$-duality pairing
which coincides with the $L_2(\mS)$-inner product when $v$ and $w$ belong to 
$L_2(\mS)$.
This bilinear form is clearly bounded and 
coercive, i.e., 
\begin{gather}
a(u,v) 
\le 
\al_2
\norm{u}{H^{1/2}(\mS)}
\norm{v}{H^{1/2}(\mS)}
\qall 
u,v\in H^{1/2}(\mS),
\label{auv cont}
\end{gather}
and 
\begin{gather}
\al_1
\norm{v}{H^{1/2}(\mS)}^2
\le
a(v,v) 
\quad
\forall
v\in H^{1/2}(\mS),
\label{equ:bi li nor}
\end{gather}
respectively. 
A natural weak formulation of 
equation~\eqref{eqn:Heqn} is:
Find $u\in H^{1/2}(\mS)$ satisfying
\begin{equation}\label{wsol1}
a(u,v) = 
\inpro{f}{v}
\qall 
v\in H^{1/2}(\mS).
\end{equation}
Let $\De$ be a spherical 
triangulation on $\mS$. 
We denote by 
$u_{\De}\in S_d^r(\De)$ the Galerkin solution 
\begin{equation}\label{GalSol1}
a(u_{\De}, v) 
=
\inprod{f}{v}
\qall 
v\in S_d^r(\De).
\end{equation}
The unique existences  of $u$ and $u_{\De}$ 
are guaranteed by the Lax--Milgram Theorem, 
noting the boundedness~\eqref{auv cont} and the
coercivity~\eqref{equ:bi li nor} of the 
bilinear form $a(\cdot,\cdot)$. Furthermore, 
if $\De$ is a {regular and quasi-uniform 
triangulation} and if $u\in H^s(\mS)$ for some 
$1/2\le s\le d+1$, then there holds
\begin{align}
\norm{u - u_{\De}}{H^{1/2}(\mS)}
\le 
\al_3\, 
h_{\De}^{s-1/2}
\norm{u}{H^s(\mS)},
\label{priori1}
\end{align}
see~\cite{PhamTranChernov11}.
Here,
$h_{\De}$ is the mesh size of $\De$, see~\eqref{hX def}, and
$\al_3$ is a positive constant depending only 
on $d$ and the smallest angle in $\De$.
The a priori error estimate~\eqref{priori1}
reveals the convergence and stability of the 
Galerkin approximation~\eqref{GalSol1}. However, 
 the upper 
bound of the error $\norm{u-u_{\De}}{H^{1/2}(\mS)}$ 
is 
given by the mesh size $h_{\De}$ and the norm 
$\norm{u}{H^s(\mS)}$ of the exact solution $u$ 
which is unknown. 
Furthermore, the quasi-uniform requirement means that one has to divide all 
spherical triangles in the current mesh whenever better accuracy is demanded.
In 
the next section, we prove a residual upper bound for the 
error $\norm{u-u_{\De}}{H^{1/2}(\mS)}$ in terms of the given right hand side 
$f$ and 
the approximate solutions $u_{\De}$ of the 
corresponding discrete 
problems. 

\section{A posteriori residual error 
estimate}\label{s:Error}

In this section, the error 
$\norm{u-u_{\De}}{H^s(\mS)}$ 
will be bounded above by an
a
posteriori residual error estimator. 
We assume that $f\in L_2(\mS)$. 
Since $S_d^r(\De)\subset H^{r+1}(\mS)$ (see~\cite{PhamTranChernov11}),
for each $u_{\De}\in S_d^r(\De)$, we have 
$\cN^* u_{\De}\in H^r(\mS) \subset L_2(\mS)$. The
\textit{residual} 
$\mathcal{R}(u_{\De})\in L_2(\mS)$
is defined by
\begin{align}\label{eq5}
\cR(u_{\De}) = f 
- 
\cN^* u_{\De}
\in L_2(\mS).
\end{align}
This together with~\eqref{auv def} and~\eqref{wsol1} gives
\begin{align}
\inprod{\mathcal{R}(u_{\De})}{v}  
&
=
\inprod{f}{v}
- 
a\left(u_{\De},v\right)=a(u-u_{\De},v)
\quad 
\forall v\in H^{1/2}(\mS).
\label{Ra u}
\end{align} 
It is obvious from~\eqref{eq5} that
the residual $\cR(u_{\De})$ depends solely on the 
source 
term $f$ and the 
discrete solution $u_{\De}$.  The following lemma 
states the equivalence of the error 
$\norm{u-u_{\De}}{H^{1/2}(\mS)}$ and the $H^{-1/2}(\mS)$-norm of the residual 
$\cR(u_{\De})$.

\begin{lemma}
Let 
$u$ and $u_{\De}$ be the weak and approximate 
solutions defined by~\eqref{wsol1} 
and~\eqref{GalSol1}, respectively.
There holds 
\begin{align*}
\al_1 
\norm{u-u_{\De}}{H^{1/2}(\mS)} 
\leq 
\norm{\mathcal{R}(u_{\De})}{H^{-1/2}(\mS)} 
\leq 
\al_2 
\norm{u-u_{\De}}{H^{1/2}(\mS)},
\end{align*} 
where $\al_1$ and $\al_2$ are the coercivity and boundedness constants, 
see~\eqref{equ:bi li nor} and~\eqref{auv cont}, respectively.
\end{lemma}

\begin{proof}
Noting~\eqref{dual1}, we have
\begin{align}
\left\| \mathcal{R}(u_{\De}) \right\|_{H^{-1/2}(\mS)} 
= 
\sup_{v \in  H^{1/2}(\mS)\atop v\not= 0} 
\frac{\left<\mathcal{R}(u_{\De}),v  
\right>}{\left\|v\right\|_{H^{1/2}(\mS)}}.
\label{RuX H norm}
\end{align}
It follows from the coercivity~\eqref{equ:bi li nor} of the 
bilinear form $a(\cdot, \cdot)$ and~\eqref{Ra u} that
\begin{align*}
\al_1\norm{u- u_{\De}}{H^{1/2}(\mS)}^2
&
\le 
a(u- u_{\De}, u-u_{\De})
=
\inpro{\cR(u_{\De})}{u - u_{\De}}.
\end{align*}
This together with~\eqref{RuX H norm} implies 
\begin{align*}
\al_1
\norm{u- u_{\De}}{H^{1/2}(\mS)}
&
\le 
\frac{ 
\inpro{\cR(u_{\De})}{u - u_{\De}}
}{
\norm{u- u_{\De}}{H^{1/2}(\mS)}
}
\leq 
\sup_{
v\in H^{1/2}(\mS)
\atop 
v\not= 0
}
\frac{ 
\inpro{\cR(u_{\De})}{v}
}{
\norm{v}{H^{1/2}(\mS)}
}
=
\norm{\cR(u_{\De})}{H^{-1/2}(\mS)}.
\end{align*}
On the 
other hand, 
we derive
\begin{align*}
\inpro{\cR(u_{\De})}{v}
&
=
a(u-u_{\De}, v)
\le
\al_2
\norm{u- u_{\De}}{H^{1/2}(\mS)}
\norm{v}{H^{1/2}(\mS)}
\qall 
v\in H^{1/2}(\mS),
\end{align*}
noting~\eqref{Ra u} and the 
continuity~\eqref{auv cont} of the bilinear form 
$a(\cdot,\cdot)$.
This implies
\begin{align*}
\norm{\cR(u_{\De})}{H^{-1/2}(\mS)}
&
=
\sup_{
v\in H^{1/2}(\mS)
\atop 
v\not= 0
}
\frac{ 
\inpro{\cR(u_{\De})}{v}
}{
\norm{v}{H^{1/2}(\mS)}
}
\le
\al_2
\norm{u- u_{\De}}{H^{1/2}(\mS)},
\end{align*}
finishing the proof of the lemma.
\end{proof}

For each $\tau\in \De$,
we define  the
\textit{spherical triangle 
residual} by
\begin{align}
\cR_\tau (u_{\De})
&
= 
\left(
f
-
\cN^* u_{\De}
\right)\big|_{\tau},
\label{Rtau def}
\end{align}
and the
\textit{local 
error estimator} $\eta_{\De,s}(\tau)$ by
\begin{align}
\eta_{\De,s}(\tau)
= 
h_\tau^{1-s} 
\left\| 
\cR_\tau(u_{\De}) 
\right\|_{L_2(\tau)},
\label{etaXtau}
\end{align}
where $0< s< 1$. 
{The residual estimators 
were used for solving the hypersingular integral equation 
with flat triangular elements, see~\cite{Carstensen2004}. In this paper, the 
local error estimators are defined on spherical triangles.
}
Note here that $f$ and $\cN^* u_{\De}$ belong to $L_2(\mS)$. It follows 
from~\eqref{eq5} and~\eqref{Rtau def} that for any $v\in 
H^{1/2}(\mS)$,  
there holds
\begin{align}
\inprod{\mathcal{R}(u_{\De})}{v}
&
= 
\int_{\mS} 
\left(
f
-
\cN^* u_{\De}
\right)v\,d\sigma
\notag
\\
&
=
\sum_{\tau\in \De}
\int_{\tau} 
\left(
f
-
\cN^* u_{\De}
\right)v\,d\sigma
\notag
\\
&
=
\sum_{\tau\in \De} 
\int_\tau 
\cR_\tau(u_{\De})\,v\,d\sigma.
\label{RuX tau v}
\end{align}

The following lemma shows an approximation property of the 
quasi--interpolation operator $Q$ (defined in Subsection~\ref{s:Q def}). This 
result extends Theorem~2 in~\cite{BarLai05} in which we relax on the  
quasi-uniform condition of $\De$.

\begin{lemma}\label{p:VQv mk}
Let $m$ be  a positive integer
satisfying 
\begin{equation}\label{m mod d1}
m=\left\{\begin{matrix}
1,3,\ldots,d+1\ \text{if d is even,} 
\\ 
2,4,\ldots,d+1 \ \text{if d is odd.}
\end{matrix}\right.
\end{equation}
Assume that $\De$ is a regular
spherical triangulation 
such that $\abs{\Om_\tau}<1$ for all $\tau\in 
\De$.
Recall that
$Q: L_2(\mS)\goto S_d^r(\De)$ is the 
quasi-interpolation operator defined by~\eqref{Qv 
def}. 
For any $\tau\in \De$, if  
$v \in H^m(\Om_\tau)$, then there holds
\begin{equation}\label{vQv regular1}
\snorm{v-Qv}{H^k(\tau)}
\le 
\al_4\, 
h_\tau^{m-k}
\snorm{v}{H^m(\Om_\tau)}.
\end{equation}
for all $k=0,\ldots,\min\sett{m-1,r+1}$.
Here, $\al_4$ is a positive constant depending 
only on $d$ and the smallest angle in $\De$.

\end{lemma}

\begin{proof}
Note here that for any $m$ satisfying~\eqref{m 
mod 
d1}, we have $d-(m-1)$ is an even number, and 
thus $\snorm{\vecx}{}^{d-(m-1)}$, for $\vecx = 
(x_1, x_2, x_3)$, is a homogeneous polynomial of 
degree $d-(m-1)$. Furthermore, for any $\vecx\in 
\mS$, we have $\snorm{\vecx}{}^{d-(m-1)} = 1$, 
and thus if $s\in \Pi_{m-1}(\mS)$, then
\[
s = s\snorm{\vecx}{}^{d-(m-1)}\in \Pi_d(\mS).
\]

By 
Theorem~4.2 in~\cite{NeaSch04}, for any $v\in H^m(\Om_\tau)$, 
there 
exists a spherical homogeneous polynomial $s\in 
\Pi_{m-1}(\mS)\subset \Pi_d(\mS)$ such that 
\begin{align}
\snorm{v- s}{H^k(\Om_\tau)}
&
\le
C_1\,
\diam(\Om_\tau)^{m-k}
\snorm{v}{H^m(\Om_\tau)}.
\label{vsmk1}
\end{align}
In particular, when $k=0$ we have 
\begin{equation}\label{k0 vs}
\norm{v-s}{L_2(\Om_\tau)}
\le 
C_1\,
\diam(\Om_\tau)^m 
\snorm{v}{H^m(\Om_\tau)}.
\end{equation}
Since $s$ is a spherical homogeneous 
polynomial of degree $d$ on $\mS$, 
Lemma~9 in~\cite{NeaSch04} assures that
$s= Q s$ 
and 
\begin{align}
\snorm{Q (v-s)}{H^k(\tau)}
&
\le 
C_2
\brac{
\tan\frac{\rho_\tau}{2}
}^{-k}
\norm{v-s}{L_2(\Om_\tau)}.
\label{Qv stau}
\end{align}
This together with~\eqref{k0 vs} implies
\begin{align}
\snorm{Q (v-s)}{H^k(\tau)}
&
\le 
C_1 \,C_2
\brac{
\tan\frac{\rho_\tau}{2}
}^{-k}
\diam(\Om_\tau)^{m}
\snorm{v}{H^m(\Om_\tau)}.
\label{Qvs 2Om}
\end{align}
Since $s = Qs$, by
using the triangle inequality and
noting~\eqref{vsmk1},~\eqref{Qvs 2Om}, we 
obtain 
\begin{align}
\snorm{v- Qv}{H^k(\tau)}
&
\le 
\snorm{v- s}{H^k(\tau)}
+
\snorm{Q(v-s)}{H^k(\tau)}
\notag
\\
&
\le
\snorm{v- s}{H^k(\Om_\tau)}
+
\snorm{Q(v-s)}{H^k(\tau)}
\notag
\\
&
\le 
C_1\,
\diam(\Om_\tau)^{m-k} 
\snorm{v}{H^m(\Om_\tau)}
\notag 
\\
&
\quad
+
C_1\,C_2 
\brac{
\tan\frac{\rho_\tau}{2}
}^{-k}
\diam(\Om_\tau)^{m}
\snorm{v}{H^m(\Om_\tau)}.
\label{vQv 12}
\end{align}
Since $\De$ is regular, the 
inequality~\eqref{vQv regular1}
is derived from~\eqref{vQv 12} and 
noting~\eqref{regular2} and~\eqref{regu local 
uniform}.

\end{proof}

The inequality~\eqref{vQv regular1} in Lemma~\ref{p:VQv mk} holds for any 
integer $m$ satisfying~\eqref{m mod d1}. In the following lemma, the inequality 
is proved when $k=0$ and $m$ is a real number between~$0$ and $1$.

\begin{lemma}\label{l:uQu1}
Let $\De$ be  a 
regular spherical triangulation such that 
$\abs{\Om_\tau}<1$ for all $\tau\in 
\De$,
and let 
$Q:L_2(\mS)\goto S_d^r(\De)$ be the 
quasi-interpolation operator defined 
by~\eqref{Qv def}.
For any 
$v\in H^{s}(\mS)$ where $0 \leq s\leq 1$, there holds
\begin{equation}\label{vQv1}
\norm{v- Q v}{L_2(\tau)}
\leq 
\al_5\,
h_\tau^{s} 
\norm{v}{H^{s}(\Om_\tau)}',
\end{equation}
where $\al_5$ is a positive constant
depending only on the smallest angle of triangles 
in $\De$.
\end{lemma}

\begin{proof}
Using the result in~\cite[Lemma 9]{BarLai05}, we 
have 
\begin{equation*}\label{Qv v norm}
\norm{Qv}{L_2(\tau)}
\le 
C
\norm{v}{L_2(\Om_\tau)},
\end{equation*}
where $C$ is a constant that depends only on the 
smallest angle of $\tau$.
This together with the triangle inequality implies
\begin{align}
\norm{v - Qv}{L_2(\tau)}
\leq 
\norm{v}{L_2(\tau)}
+
\norm{Qv}{L_2(\tau)}
\le 
(1+C)
\norm{v}{L_2(\Om_\tau)}.
\label{QvvL2}
\end{align}
If $d$ is 
even, we apply Lemma~\ref{p:VQv mk} when 
$k=0$ and $m=1$ to obtain
\begin{equation}\label{vQv 1even}
\norm{v- Q v}{L_2(\tau)}
\leq 
\al_4\,
h_\tau 
\snorm{v}{H^1(\Om_\tau)}
\leq 
\al_4\,
h_\tau 
\norm{v}{H^1(\Om_\tau)}'.
\end{equation}
Noting~\eqref{QvvL2}, \eqref{vQv 1even} and 
using~\cite[Theorem B.2]{McL00} (for $\theta = 
s$ where $0 \leq s\leq 1$), we obtain
\begin{align}
\norm{v- Q v}{L_2(\tau)}
&
\leq 
(1+C)^{1-s}\,
\al_4^{s}\,
h_\tau^{s}\,
\norm{v}{H^{s}(\mS)}',
\notag
\end{align}
proving~\eqref{vQv1} when $d$ is even.
We now prove~\eqref{vQv1} when  $d$ 
is odd. 
Applying 
Lemma~\ref{p:VQv mk} when $k=0$ and $m=2$ 
we have
\begin{equation}\label{QvvL2a}
\norm{v - Qv}{L_2(\tau)}
\leq 
\al_4\, 
h_\tau^2
\snorm{v}{H^2(\Om_\tau)}
\leq 
\al_4\,
h_\tau^2
\norm{v}{H^2(\Om_\tau)}'.
\end{equation}
Noting~\eqref{QvvL2}, \eqref{QvvL2a} and  
applying~\cite[Theorem 
B.2]{McL00} (for $\theta = s/2$ where $ 0 \leq s\leq 2$) 
we obtain 
\[
\norm{v - Qv}{L_2(\tau)}
\leq 
(1+C)^{(1-s/2)}\, \al_4^{s/2}\, 
h_\tau^{s}
\norm{v}{H^s(\Om_\tau)}',
\]
completing the proof of the lemma.
\end{proof}

Technical results in the following two lemmas will be used in the proof of 
Theorem~\ref{t:upp bound}.

\begin{lemma}\label{De Om tau card}
Let $\De$ be a regular spherical triangulation on the unit sphere. There 
holds 
\begin{equation}\label{Om tau tau ij}
\cardd
\sett{ 
\tau'\in \De:
\intt\Om_{\tau'}
\cap 
\intt\Om_{\tau} \not= \emptyset
}
\le 
\al_6
\quad 
\forall 
\tau\in \De,
\end{equation}
where $\al_6$ is a positive constant which depends only on the smallest angle 
of $\De$.
\end{lemma}

\begin{proof}
Noting~\eqref{Tvi bounded} 
there holds
\begin{align}
\cardd T_{\vv}^{\De}
\le {L}
\quad 
\forall 
\vv\in V_{\De},
\notag
\end{align}
where $L$ is a positive constant depending only on the smallest angle of 
$\De$. If $\vv_1$, 
$\vv_2$ and $\vv_3$ are the vertices of $\tau$ then 
$\Om_{\tau} = 
\bigcup\sett{\wtd\tau\in T_{\vv_i}^{\De}, i = 1,2,3}$ and thus 
\begin{align}
\cardd
\sett{\wtd\tau\in \De: \tau \subset \Om_{\tau}}
\le 
3{L.}
\label{card tau Om}
\end{align}
Suppose that $\tau'\in \De$ satisfies 
$\intt\Om_{\tau'}
\cap 
\intt\Om_{\tau} \not= \emptyset$. Then, there is a $\wtd\tau\in \De$ such 
that $\wtd\tau\subset \Om_{\tau'}\cap\Om_{\tau}$.
If $\wtd\tau\subset \Om_{\tau'}$ then $\tau'\subset \Om_{\wtd\tau}$. 
For every $\tau\in \De$, there are at most $3L$ 
options of choosing a $\wtd\tau\subset \Om_{\tau}$ by~\eqref{card tau Om}. On 
the other hand, for each $\wtd\tau$ in $\Om_{\tau}$, there are at most $3L$ 
options of choosing  
a $\tau'\subset \Om_{\wtd\tau}$. Thus, there holds
\begin{align}
\cardd
\sett{ 
\tau'\in \De:
\intt\Om_{\tau'}
\cap 
\intt\Om_{\tau} \not= \emptyset
}
\le 
{9L^2}
\quad 
\forall 
\tau\in \De.
\notag
\end{align}
Denoting $\al_6 = {9L^2}$, we obtain the inequality~\eqref{Om tau tau ij}, 
completing the proof of the lemma.
\end{proof}

\begin{lemma}\label{l:coloring}
Let $\De$ be a regular spherical triangulation and 
let $s\in [0,1]$. There exists a positive number $\al_7$ which depends only 
on 
the smallest angle of $\De$ such that 
\begin{align}
\sum_{\tau\in \De}
\norm{v}{H^s(\Om_\tau)}'^2
\le
\al_7 
\norm{v}{H^s(\mS)}^2
\quad 
\forall 
v\in H^s(\mS).
\label{sum norm S tau}
\end{align}
\end{lemma}

\begin{proof}
Since $\De$ is regular, by applying Lemma~\ref{De Om tau card}, there holds
\begin{align}
\max 
\sett{ 
\cardd
\sett{ 
\tau'\in \De:
\intt\Om_{\tau'}
\cap 
\intt\Om_{\tau} \not= \emptyset
}:
\tau\in \De
}
\le 
\al_6.
\label{max card int}
\end{align}
The set $\sett{\intt\Om_{\tau}: \tau\in \De}$ is a set of overlapping subsets 
which covers the unit sphere $\mS$. 
The coloring argument (see 
e.g.~\cite{CarstensenMaischakStephan01})
suggests that the set $\sett{\intt\Om_{\tau}: \tau\in \De}$ can be divided 
into $C_1$ groups 
\[
\sett{\intt\Om_{\tau}: \tau\in I_k}, 
\quad 
k=1,\ldots, C_1,
\]
so that each group consists of mutually disjoint subsets. Here, the constant 
$C_1$ satisfies 
\[
C_1 
\le 
\max 
\sett{ 
\cardd
\sett{ 
\tau'\in \De:
\intt\Om_{\tau'}
\cap 
\intt\Om_{\tau} \not= \emptyset
}:
\tau\in \De
}.
\]
Since $\intt\Om_{\tau}\cap \intt\Om_{\tau'} = \emptyset$ if $\tau$ and $\tau'$ 
are two triangles that belong to the set $I_k$ and $\bigcup\sett{\Om_{\tau}: 
\tau\in I_k} \subset \mS$, there holds
\[
\sum_{\tau\in I_k}
\norm{v}{H^s(\Om_{\tau})}'^2
\le 
\norm{v}{H^s(\mS)}'^2, 
\quad 
k=1,\ldots, C_1,
\]
see~\cite{Carstensen2004, Pet}. We obtain 
\[
\sum_{\tau\in \De} 
\norm{v}{H^s(\Om_{\tau})}'^2
=
\sum_{k=1}^{C_1} 
\sum_{\tau\in I_k}
\norm{v}{H^s(\Om_{\tau})}'^2
\le 
C_1 
\norm{v}{H^s(\mS)}'^2
\le 
C_1 
\ga_3^2 
\norm{v}{H^s(\mS)}^2
\]
noting~\eqref{norms euiv}.
The inequality~\eqref{sum norm S tau} can then be derived by denoting $\al_7 = 
C_1 \ga_3^2$, completing the proof of the lemma.
\end{proof}

Recalling the local error estimator $\eta_{\De,s}(\tau)$ (see~\eqref{etaXtau}), 
for a 
subset $\Om\subset \mS$, we define the \textit{error estimator} 
$\eta_{\De,s}(\Om)$ 
by 
\[
\eta_{\De,s}(\Om) 
=
\brac{ 
\sum_{\tau\in \De
\atop 
\tau \cap \Om \not= \emptyset
}
\eta_{\De,s}(\tau)^2
}^{1/2}.
\]
In particular, we denote by $\eta_{\De,s}(\mS)$ the residual-type 
error estimator with respect to the mesh 
$\De$, i.e.,
\begin{equation}\label{etaS}
\eta_{\De,s}(\mS) 
=
\left(
\sum_{\tau\in\De} 
\eta_{\De,s}(\tau)^2
\right)^{1/2}.
\end{equation}

We are now ready to prove the main result of this section.
The error $\norm{u-u_{\De}}{H^s(\mS)}$ will be bounded above by the residual 
error 
estimator $\eta_{\De,s}(\mS)$.

\begin{theorem}[A posteriori residual upper 
bound]\label{t:upp bound}
Let $\De$ be  a 
regular spherical triangulation such that 
$\abs{\Om_\tau}<1$ for all $\tau\in 
\De$.
Let 
$u$ and $u_{\De}$ be the weak and approximate 
solutions defined by~\eqref{wsol1} 
and~\eqref{GalSol1}, respectively.
There  
exists a 
positive constant $\al_8$ depending only on 
the 
smallest angle of~$\De$
such that for all $0 \leq s\leq 1/2$
\begin{align}
\norm{u-u_{\De}}{{H^{s}(\mS)}}
\leq 
\al_8\, 
\eta_{\De,s}(\mS).
\label{u ux eta Xs}
\end{align}
Here, $\al_8$ is a positive constant depending 
only on $d$ and the smallest angle of $\De$.
\end{theorem}

\begin{proof}
Employing~\eqref{Ra u}, \eqref{wsol1} and~\eqref{GalSol1}, we derive 
\begin{align}
\inpro{\cR(u_{\De})}{v}
&
=
a(u-u_{\De},v) 
=
0
\quad 
\forall 
v\in S_d^r(\De).
\label{R def}
\end{align}
Using the duality argument~\eqref{equ:neg nor} and noting \eqref{R def}, we 
obtain
\begin{align*}
\norm{\cR(u_{\De})}{H^{s-1}(\mS)}
&
=
\sup_{v \in H^{1-s}(\mS)
\atop 
v\not= 0}
\frac{\inpro{\cR(u_{\De})}{v}
}{\norm{v}{H^{1-s}(\mS)}}\\
&=
\sup_{v \in H^{1-s}(\mS)
\atop 
v\not= 0
}
\frac{\inpro{\cR(u_{\De})}{v-Qv}
}{\norm{v}{H^{1-s}(\mS)}}.
\end{align*}
Note that $v\in H^{1-s}(\mS) \subset L_2(\mS)$ for every $s\in [0,1/2]$,
and
$Qv\in S_d^r(\De)\subset H^{1}(\mS)$.
By~\eqref{RuX tau v}, we have 
\begin{align*}
\norm{\cR(u_{\De})}{H^{s-1}(\mS)}
&=
\sup_{v \in H^{1-s}(\mS)
\atop 
v\not= 0
}
\frac{
\displaystyle
\sum_{\tau\in \De}
\int_\tau \cR(u_{\De})\,(v-Qv) d\sigma
}{\norm{v}{H^{1-s}(\mS)}}\\
&
\leq
\sup_{v \in H^{1-s}(\mS)
\atop 
v\not= 0
}
\frac{
\displaystyle
\sum_{\tau\in \De}
\norm{\cR(u_{\De})}{L_2(\tau)}
\norm{v-Qv}{L_2(\tau)}
}{\norm{v}{H^{1-s}(\mS)}},
\end{align*}
where in the second step we apply Cauchy-Schwarz inequality. This together 
with 
the result in Lemma~\ref{l:uQu1} gives
\begin{align*}
\norm{\cR(u_{\De})}{H^{s-1}(\mS)}
&
\leq
\al_5
\sup_{v \in H^{1-s}(\mS)
\atop v\not= 0}
\frac{
\displaystyle
\sum_{\tau\in \De}
h_\tau^{1-s}
\norm{\cR(u_{\De})}{L_2(\tau)}
\norm{v}{H^{1-s}(\Om_\tau)}'
}{\norm{v}{H^{1-s}(\mS)}}.
\end{align*}
By using Cauchy-Schwarz inequality and applying Lemma~\ref{l:coloring}, we have
\begin{align*}
\norm{\cR(u_{\De})}{H^{s-1}(\mS)}
&
\leq
\al_5
\brac{
\sum_{\tau\in \De}
h_\tau^{2-2s}
\norm{\cR(u_{\De})}{L_2(\tau)}^2
}^{1/2}
\sup_{v \in H^{1-s}(\mS)
\atop 
v\not= 0
}
\frac{
\brac{
\displaystyle
\sum_{\tau\in \De}
\norm{v}{H^{1-s}(\Om_\tau)}'^2
}^{1/2}
}{\norm{v}{H^{1-s}(\mS)}}\\
&\leq
\al_5
\sqrt{\al_7}
\brac{
\sum_{\tau\in \De}
h_\tau^{2-2s}
\norm{\cR(u_{\De}) }{L_2(\tau)}^2
}^{1/2}.
% \\
% &
% =
% \al_7\, 
% \eta_{\De,s}(\mS),
\end{align*}
Noting~\eqref{eq5}, \eqref{N u om f S} and~\eqref{N star def}, we have 
$\cR(u_{\De}) = \cN^*(u-u_{\De})$.
{Since $0\le s\le 1/2$, we have $u\in H^{1/2}(\mS) \subset H^s(\mS)$.}
Applying the inequality~\eqref{Nv norm v sm1} and noting~\eqref{etaS} 
and~\eqref{etaXtau}, we obtain 
\[
\norm{u-u_{\De}}{H^s(\mS)}
\le 
\al_1^{-1} \al_5 
\sqrt{\al_7}\,  
\eta_{\De,s}(\mS),
\]
finishing the proof of the theorem.
\end{proof}

\section{An a posteriori hierarchical error 
estimation}\label{s:hierarchical 
sec}

Hierarchical basis techniques have been used to prove a posteriori error 
estimates when solving hypersingular integral equation in two dimensions and 
linear elements, see e.g.~\cite{Maischak1997, MunSte00, Carstensen2007, 
Erath2013}.
In this section, we discuss the use of these techniques to prove an a 
posteriori upper bound for the error 
$\norm{u-u_\De}{H^{1/2}(\mS)}$ when solving the hypersingular integral equation 
on the unit sphere,
where the approximate solution $u_{\De}$ is found in the space 
$S_1^0(\De)$
and $\De$ is a spherical triangulation on $\mS$.
In the remainder of 
this paper,
we use $S(\De)$ instead of 
$S_1^0(\De)$ for notational convenience.
Suppose that the set $V_{\De} = \sett{\vecv_1,\vecv_2,\ldots, \vecv_M}$
is the set of all vertices of $\De$.
For each vertex $\vecv_i$, the associated basis function $B_{\vecv_i}$ is 
defined by
\begin{equation}\label{Bvi def}
B_{\vecv_i}(\vecx)
=
\begin{cases}
0 
&
\text{if }
\vecx\notin 
\bigcup\sett{\tau: \tau\in T_{\vecv_i}^{\De}}
\\
b_{1,\tau}(\vecx)
&
\text{if }
\vecx\in\tau = \langle \vecv_i,\vecv_j,\vecv_k\rangle \in T_{\vecv_i}^{\De},
\end{cases}
\end{equation}
where $b_{1,\tau}(\vx)$ is the first spherical barycentric coordinate of $\vx$ 
with respect to $\tau$, see~\eqref{b1tau def} and~\eqref{b1v deter}.
We then have 
\[
S(\De)
=
\spann
\sett{B_{\vecv_1}, B_{\vecv_2}, \ldots, B_{\vecv_M}}.
\]
Recalling the definition of the quasi-interpolation operator with respect to 
the space $S_d^r(\De)$ in Subsection~\ref{s:Q def},
the quasi-interpolation operator: $Q: L_2(\mS) 
\goto S(\De)$ is given by
\begin{equation}\label{Qv def Hier}
Qv=\sum_{i=1}^M \nu_{\vecv_i}(v) B_{\vecv_i}, \quad 
v\in L_2(\mS).
\end{equation}
Here, $\nu_{\vecv_i}(v) = v(\vecv_i)$ for all $v\in S(\De)$.
The quasi-interpolation operator is a projection onto $S(\De)$, i.e. $Q^2 v = 
Qv$ for every $v\in L_2(\mS)$.
Every $s\in S(\De)$ can uniquely be written as 
\[
s 
=
\sum_{i=1}^M 
\nu_{\vecv_i}(s) 
B_{\vecv_i},
\quad 
\text{where }
\nu_{\vv_i}(s) = s(\vv_i).
\]

\begin{lemma}\label{BiX lem}
Let $\De$ be a regular spherical triangulation such that $h_{\tau}<1$ for all 
$\tau\in \De$. 
For every vertex $\vecv$ and any $\tau\in T_{\vecv}^{\De}$,
the basis function $B_{\vecv}\in S(\De)$ associated with
$\vecv$ (see~\eqref{Bvi def}) satisfies 
\begin{align}\label{BiX ine}
\norm{B_{\vecv}}{H^{1/2}(\mS)}
\leq
\al_9 h_\tau^{1/2},
\end{align}
where 
 $\al_9$ is a constant which depends only on 
the smallest angle of $\De$.
\end{lemma}

\begin{proof}
Since $\De$ is regular, the cardinality of $T_{\vecv}^{\De}$ is bounded, i.e.,
$\textrm{card}\brac{T_{\vecv}^{\De}}
\le L$
for some positive integer $L$ depending only on the smallest angle of $\De$, 
see~\eqref{Tvi bounded}. 
If $\tau, \tau'\in T_{\vv}^{\De}$,  then $\tau'\subset 
\Om_{\tau}$. This together with~\eqref{regu local uniform} implies 
\[
h_{\tau'}
\le 
C_1 
h_{\tau}
\quad 
\forall 
\tau'\in T_{\vv}^{\De},
\]
for some positive constant $C_1$ depending only on the smallest angle in $\De$. 
We then have 
\begin{equation}\label{tau taumax v}
\max\sett{h_{\tau'}: \tau' \in T_{\vecv}^{\De}}
\le 
C_1
h_{\tau}
\quad 
\forall 
\tau\in T_{\vecv}^{\De}.
\end{equation}
Statement~(5) in~\cite[Proposition 5.1]{NeaSch04} and~\eqref{Atau htau 
regular} give
\begin{align}
\norm{B_{\vecv}}{L_2(\tau')}
&
\le 
C_2
A_{\tau'}^{1/2}
\le 
C_2 \beta_4 h_{\tau'}
\quad 
\forall \tau'\in T_{\vecv}^{\De}.
\notag
% \label{Bi 12a}
\end{align}
Since $\supp 
B_{\vecv} \subset \bigcup\sett{\tau': \tau'\in T_{\vecv}^{\De}}$ and 
noting~\eqref{tau taumax v}, we obtain
\begin{align}
\norm{B_{\vecv}}{L_2(\mS)} 
&
\le 
\sqrt{L} C_2 \beta_4 
\max
\sett{
h_{\tau'}:
\tau'\in T_{\vecv}^{\De}
}
\le 
\sqrt{L} C_1 
C_2\beta_4 
h_{\tau}= C_3 h_{\tau}.
\label{Bv S htau}
\end{align}
Similarly, the
inequality (8) in~\cite[Proposition 5.1]{NeaSch04} together 
with~\eqref{regular2} and~\eqref{Atau 
htau regular}
yields 
\begin{align}
\snorm{B_{\vecv}}{H^1(\tau')}
&
\le 
C_4 
\rho_{\tau'}^{-1} A_{\tau'}^{1/2}
\le 
C_4\beta_2 \beta_4 = C_5.
\notag
% \label{Bi 12b}
\end{align}
Since this is true for all $\tau'\in T_{\vecv}^{\De}$ and $\supp B_{\vecv} = 
\bigcup\sett{\tau': \tau'\in T_{\vecv}^{\De}}$, we have
\begin{align}
\snorm{B_{\vecv}}{H^1(\mS)}
&
\le 
\sqrt{L} C_5.
\label{BvH1}
\end{align}
On the other hand,
the size $h_{\tau}$ is smaller than $1$ for every $\tau\in \De$. This 
together with~\eqref{Bv S htau} implies
\begin{equation*}\label{Bv C4}
\norm{B_{\vecv}}{L_2(\mS)}
\le 
C_3.
\end{equation*}
This together with~\eqref{BvH1} 
implies
\begin{align}
\norm{B_{\vecv}}{H^1(\mS)}'
\le 
% \brac{C_3^2 + N C_5^2}^{1/2} = C_6.
C_3 + \sqrt{L} C_5 = C_6.
\label{Bv C5}
\end{align}
Noting~\eqref{Bv S htau}, \eqref{Bv C5} and applying the interpolation 
inequality (see e.g.~\cite[Lemma B.1]{McL00}), we derive 
\begin{align}
\norm{B_{\vecv}}{H^{1/2}(\mS)}'
&
\le
\norm{B_{\vecv}}{L_2(\mS)}^{1/2}
\norm{B_{\vecv}}{H^1(\mS)}^{1/2}
\le 
\sqrt{C_3 C_6}\, 
h_{\tau}^{1/2}.
\notag
\end{align}
This together with~\eqref{norms euiv} yields~\eqref{BiX ine} where $\al_9 = 
\ga_3\sqrt{C_3 C_6}$,
completing the proof of this lemma.
\end{proof}

A spherical triangulation $\De'$ is said to be a refinement of another 
spherical triangulation $\De$ if 
every spherical triangle $\tau'\in \De'$ is a subtriangle of a triangle 
$\tau\in \De$.
When $\De'$ is a refinement of $\De$, we call $\De$ a coarser triangulation 
(coarser mesh) and $\De'$ is a finer triangulation (finer mesh). In this case, 
the two 
spherical triangulations are said to be nested.

Suppose that $\De$ and $\De'$ are two nested spherical triangulations, where 
$\De'$ is the finer mesh.  Then the space $S(\De)$ is a subspace of $S(\De')$.
We denote by $u_{\De}$ and $u_{\De'}$ the Galerkin solutions to the 
hypersingular integral equation~\eqref{N u om f S}, i.e., $u_{\De}\in 
S(\De)$ and 
$u_{\De'}\in S(\De')$ satisfy
\begin{equation}\label{uk wsol}
a(u_{\De}, v) 
=
\inpro{f}{v}
\quad 
\forall 
v\in S(\De),
\end{equation}
and 
\begin{equation}\label{ukp wsol}
a(u_{\De'}, v) 
=
\inpro{f}{v}
\quad 
\forall 
v\in S(\De').
\end{equation}
Following e.g.~\cite{Erath2013,Maischak1997, MunSte00}, 
we assume that the two triangulations $\De$ and $\De'$ satisfy the saturation 
assumption:
\begin{equation}\label{satur assumption}
\norm{u - u_{\De'}}{H^{1/2}(\mS)}
\le 
\eta 
\norm{u - u_{\De}}{H^{1/2}(\mS)}
\end{equation}
for some fixed $\eta\in (0,1)$. Here, the function $u$ in~\eqref{satur 
assumption} is the weak solution to the hypersingular integral equation defined 
by~\eqref{wsol1}. In our adaptive refinement strategy which will be discussed 
in Section~\ref{s:mesh ref}, {the approximate solution $u_{\De'}$ is not 
computed and the finer mesh $\De'$ only plays a role as a mean to evaluate only 
local error estimators which will then be used to 
conduct mesh refinement step and create
better 
approximation 
spaces}. In our numerical experiments (Section~\ref{s:num ex}), $\De'$ is  
created from $\De$ by joining midpoints of the three spherical edges in each 
spherical triangle of $\De$.

In this section, 
for each vertex $\vv_i\in V_\De$, 
we denote by $B_{\vecv_i}$ 
the hat function in $S(\De)$ 
corresponding to the 
vertex $\vv_i$, see~\eqref{Bvi def}. 
Since $V_{\De}\subset V_{\De'}$, 
the vertex $\vv_i$ is also a vertex in the spherical triangulation $\De'$. 
If $\vv_i\in V_{\De'}$,
we 
denote by $B_{\vv_i}'$
the hat function in $S(\De')$ associated with the vertex $\vv_i$.
We recall here that  $Q_\De$ and $Q_{\De'}$ denote the 
quasi-interpolation operators 
associated 
with the spaces $S(\De)$ and $S(\De')$, respectively. 
For each $\vv_i\in V_{\De'}$, we define a nodal estimator
\begin{align}
\mu_{\vv_i}
&
=
\frac{\inpro{\cR(u_{\De})}{B_{\vv_i}'}}{\norm{B_{\vv_i}'}{H^{1/2}(\mS)}},
\label{mu kp1 def}
\end{align}
 where $\cR(u_{\De})=f- \cN^*u_{\De}\in H^{-1/2}(\mS)$. 
 
\begin{lemma}\label{QX w = 0}
Let $\De$ and $\De'$ be two nested spherical triangulations where $\De'$ is the 
finer mesh. 
For every $v\in S(\De')$, we denote
\begin{align}
{
I_{\De} v 
:= 
\sum_{\vv_i\in V_{\De}}v(\vv_i) B_{\vv_i}.}
\label{May18b}
\end{align}
Suppose that $v\in S(\De')$ satisfies 
{$I_{\De}v = 0$,}
there holds
\begin{align}
v
&
=
\sum_{\vv_i\in V_{\De'}\setminus V_{\De}}
\nu_{\vv_i}'(v) B_{\vv_i}'.
\label{HyMul 1c}
\end{align}
Here, $\nu_{\vv_i}'$ is  the linear functional which picks the coefficient 
associated with vertex $\vecv_i\in V_{\De'}$. 
\end{lemma}
\begin{proof}
Since $v\in S(\De')$, $v$ can uniquely be 
 written as 
\begin{align}
v
=
\sum_{\vv_j\in V_{\De'}}
\nu_{\vv_j}'(v) B_{\vv_j}',
\quad 
\text{where }
\nu_{\vv_j}'(v)
=
v(\vv_j)
\quad 
\text{for all }
\vv_j\in V_{\De'}.
\label{HyMul 1b}
\end{align}
It follows that 
\begin{align}
I_{\De} v
&
=
\sum_{\vv_i\in V_{\De}} v(\vv_i) B_{\vv_i} 
=
\sum_{\vv_i\in V_{\De}} 
\brac{
\sum_{\vv_j\in V_{\De'}}
\nu_{\vv_j}'(v) B_{\vv_j}'
} 
(\vv_i) B_{\vv_i} 
\notag
\\
&
=
\sum_{\vv_i\in V_{\De}} 
\sum_{\vv_j\in V_{\De'}}
\nu_{\vv_j}'(v) B_{\vv_j}'
(\vv_i) B_{\vv_i}. 
\label{HyMul 1a}
\end{align}
On the other hand, we have 
\begin{align}
B_{\vv_j}'(\vv_i)
=
\begin{cases}
1 
&
\text{if }
\vv_i = \vv_j
\\
0
&
\text{if }
\vv_i \not= \vv_j.
\end{cases}
\notag
\end{align}
{This together with~\eqref{HyMul 1a} yields
\begin{align}
I_{\De}v
=
\sum_{\vv_i\in V_{\De}}
\nu_{\vv_i}'(v) B_{\vv_i}.
\notag
\end{align}
}
Since $I_{\De} v
= 0$, there holds 
\[
\sum_{\vv_i\in V_{\De}}
\nu_{\vv_i}'(v) B_{\vv_i}
=
0.
\]
This yields 
\begin{align}
\nu_{\vv_i}'(v) 
=
0
\quad 
\forall 
\vv_i\in V_{\De}.
\label{HyMul 1d}
\end{align}
Equalities~\eqref{HyMul 1b} and~\eqref{HyMul 1d} imply~\eqref{HyMul 1c}, 
completing the proof of this lemma.
\end{proof}

\begin{lemma}\label{l:e uDe De}
Let $u_{\De}$ and $u_{\De'}$ be Galerkin solutions defined by~\eqref{uk wsol} 
and~\eqref{ukp wsol}, respectively. Denote 
$e := u_{\De'} - u_{\De}$ and $w:= e - I_{\De} e$.
There holds 
\begin{align}
\inpro{\cN^* e}{e}
=
\sum_{\vv_i\in V_{\De'}{\setminus}V_{\De}}
\nu_{\vv_i}'(w)
\inp{\cR(u_{\De})}{B_{\vv_i}'}.
\notag
\end{align}
\end{lemma}

\begin{proof}
Recall that $u_{\De}$ and $u_{\De'}$ are the Galerkin solutions in the spaces 
$S(\De)$ and $S(\De')$, respectively. Noting~\eqref{uk wsol}, \eqref{ukp wsol} 
and~\eqref{auv def},
we have 
\begin{align}
\inpro{\cN^* u_{\De}}{v}
=
\inpro{f}{v}
\quad 
\forall 
v\in S(\De),
\label{HyMul 2e}
\end{align}
and 
\begin{align}
\inpro{\cN^* u_{\De'}}{v}
=
\inpro{f}{v}
\quad 
\forall 
v\in S(\De').
\label{HyMul 2f}
\end{align}
Noting that $I_{\De}$ is a 
projection, i.e., 
$(I_{\De})^2 = 
I_{\De}$,
we obtain
\begin{align}
I_{\De} w_{} 
&
=
I_{\De}(e_{} - I_{\De} e_{})
=
I_{\De} e_{} - I_{\De} e_{} 
=
0.
\label{May18a}
\end{align}
Noting that $S(\De)\subset S(\De')$, we have 
$e = u_{\De'}-u_{\De}\in S(\De')$ and $w = e - I_{\De}e \in S(\De')$. 
This together with~\eqref{May18a} and the result in Lemma~\ref{QX w = 0}
implies
\begin{align}
w_{} 
&
=
\sum_{\vv_i\in V_{\De'}{\setminus}V_{\De}}
\nu_{\vv_i}'(w_{}) B_{\vv_i}'.
\label{HyMul 2c}
\end{align}
By~\eqref{HyMul 2e} and~\eqref{HyMul 2f},
we have 
\begin{align}
\inpro{\cN^*e_{}}{I_{\De}e_{}}
&
=
\inpro{\cN^*u_{\De'}}{I_{\De}e_{}}
-
\inpro{\cN^*u_{\De}}{I_{\De}e_{}}
\notag
\\
&
=
\inpro{f}{I_{\De} e_{}}
-
\inpro{f}{I_{\De} e_{}}
\notag
\\
&
=
0,
\label{HyMul 2h}
\end{align}
noting that $I_{\De} e_{}\in S(\De)\subset S(\De')$.
It follows from~\eqref{HyMul 2h} and~\eqref{HyMul 2c} that
\begin{align}
\inpro{\cN^*e_{}}{e_{}}
&
=
\inpro{\cN^*e_{}}{e_{} - I_{\De} e_{}}
=
\inpro{\cN^*e_{}}{w_{}}
\notag
\\
&
=
\inpro{\cN^*e_{}}{
\sum_{\vv_i\in V_{\De'}{\setminus}V_{\De}}
\nu_{\vv_i}'(w_{}) B_{\vv_i}'
}
\notag
\\
&
=
\sum_{\vv_i\in V_{\De'}{\setminus}V_{\De}}
\nu_{\vv_i}'(w_{})
\inpro{\cN^* e_{}}{B_{\vv_i}'}.
\notag
\end{align}
By the definition of $e$ and by using~\eqref{HyMul 2f} (noting that 
$B_{\vv_i}'\in S(\De')$), we obtain 
\begin{align}
\inpro{\cN^* e}{e}
&
=
\sum_{\vv_i\in V_{\De'}{\setminus}V_{\De}}
\nu_{\vv_i}'(w_{})
\inpro{\cN^* (u_{\De'}-u_{\De})}{B_{\vv_i}'}
\notag
\\
&
=
\sum_{\vv_i\in V_{\De'}{\setminus}V_{\De}}
\nu_{\vv_i}'(w_{})
\inpro{f-\cN^*u_{\De}}{B_{\vv_i}'}
\notag
\\
&
=
\sum_{\vv_i\in V_{\De'}{\setminus}V_{\De}}
\nu_{\vv_i}'(w_{})
\inpro{\cR(u_{\De})}{B_{\vv_i}'},
\notag
\end{align}
completing the proof of this lemma.

\end{proof}

We are now ready to prove the main theorem of this section, an upper bound for 
the error
$\norm{u - u_{\De}}{H^{1/2}(\mS)}$ in terms of the error estimators 
$\mu_{\vv_i}$, see~\eqref{mu kp1 def}.

\begin{theorem}[A posteriori hierarchical upper bound]\label{t:main satur}
Let $\De$ and $\De'$ be two nested spherical triangulations 
(where $\De'$ is the 
finer mesh)
satisfying the saturation assumption~\eqref{satur assumption}.
There exists a positive number $\al_{10}$ 
depending only on the smallest angle of 
the triangulations and the saturation assumption constant $\eta$ 
(see~\eqref{satur assumption}) such that 
\begin{align}
\norm{u - u_{\De}}{H^{1/2}(\mS)}^2
&
\le 
\al_{10}
\sum_{\vv_i\in V_{\De'}{\setminus}V_{\De}}
\brac{\mu_{\vv_i}}^2.
\label{HyMul 2a}
\end{align}
where $\mu_{\vv_i}$ are the nodal estimators defined by~\eqref{mu kp1 def}.
\end{theorem}

\begin{proof}
The triangle inequality and the saturation assumption~\eqref{satur assumption} 
give 
\begin{align}
\norm{u- u_{\De}}{H^{1/2}(\mS)}
&
\le 
\norm{u- u_{\De'}}{H^{1/2}(\mS)}
+
\norm{u_{\De'}- u_{\De}}{H^{1/2}(\mS)}
\notag
\\
&
\le 
\eta
\norm{u- u_{\De}}{H^{1/2}(\mS)}
+
\norm{u_{\De'}- u_{\De}}{H^{1/2}(\mS)}.
\notag
\end{align}
It follows that 
\begin{align}
\norm{u- u_{\De}}{H^{1/2}(\mS)}
&
\le 
(1-\eta)^{-1}
\norm{u_{\De'}- u_{\De}}{H^{1/2}(\mS)}.
\label{HyMul 2d}
\end{align}
Suppose that $e = u_{\De'} - u_{\De}$ and $w = e- I_{\De} e$ as defined in 
Lemma~\ref{l:e uDe De}. Then we have 
\begin{align}
\inpro{\cN^* e}{e}
=
\sum_{\vv_i\in V_{\De'}{\setminus}V_{\De}}
\nu_{\vv_i}'(w)
\inp{\cR(u_{\De})}{B_{\vv_i}'}.
\label{cNe e sum}
\end{align}
Applying Statement (4) in \cite[Proposition 5.1]{NeaSch04} and~\eqref{Atau htau 
regular}, there exists a constant $C_1>0$ 
depending only on the smallest angle in $\De'$ such that 
\begin{align}
\abs{\nu_{\vv_i}'(w_{})}
\le 
C_1 
h_{\tau_i}^{-1}
\norm{w_{}}{L^2(\tau_i)}
\label{HyMul 3a}
\end{align}
for every vertex $\vv_i\in V_{\De'}$ and for every $\tau_i\in \De'$. 
Using~\eqref{HyMul 3a} and the triangle inequality, we obtain  
\begin{align}
\abs{\nu_{\vv_i}'(w_{})} 
&
\le 
C_1 h_{\tau_i}^{-1}
\norm{e - I_{\De} e}{L^2(\tau_i)}
\notag
\\
&
\le 
C_1 h_{\tau_i}^{-1}
\brac{ 
\norm{e - Q_{\De} e}{L^2(\tau_i)}
+
\norm{Q_{\De}e - I_{\De} e}{L^2(\tau_i)}
}
\notag
\\
&
=
C_1 h_{\tau_i}^{-1}
\brac{ 
\norm{e - Q_{\De} e}{L^2(\tau_i)}
+
\norm{I_{\De}(Q_{\De}e - e)}{L^2(\tau_i)}
},
\label{18d}
\end{align}
noting that $Q_{\De}e = I_{\De}(Q_{\De}e)$. It follows from~\eqref{May18b}, 
\eqref{HyMul 1b} and the triangle inequality that
\begin{align}
\norm{I_{\De}(Q_{\De}e - e)}{L^2(\tau_i)}
&
=
\Big{\|}
\sum_{\vv_j\in V_{\De}}
(Q_{\De}e - e)(\vv_j)
B_{\vv_j}
\Big{\|}_{L^2(\tau_i)}
\notag
\\
&
=
\Big{\|}
\sum_{\vv_j\in V_{\De}
\atop 
\vv_j\in \tau_i}
\nu_{\vv_j}'
(Q_{\De}e - e)
B_{\vv_j}
\Big{\|}_{L^2(\tau_i)}
\notag
\\
&
\le 
\sum_{\vv_j\in V_{\De}
\atop 
\vv_j\in \tau_i}
\abs{\nu_{\vv_j}'
(Q_{\De}e - e)
}
\norm{B_{\vv_j}}{L^2(\tau_i)}.
\label{May18c}
\end{align}
Applying Proposition~5.1 (statement (4)) 
in~\cite{NeaSch04} and~\eqref{Atau htau regular} again,
we have
\begin{align}
\abs{\nu_{\vv_j}'
(Q_{\De}e - e)
}
\le 
C_1
h_{\tau_i}^{-1}
\norm{Q_{\De}e - e}{L^2(\tau_i)}.
\label{May18e}
\end{align}
Statement (5) in~\cite[Proposition~5.1]{NeaSch04} and~\eqref{Atau htau regular} 
give 
\begin{align}
\norm{B_{\vv_j}}{L^2(\tau_i)}
\le 
C_2 h_{\tau_i}
\label{May18f}
\end{align}
for some positive number $C_2>0$ depending only on the smallest angle of $\De'$.
The inequalities~\eqref{18d}--\eqref{May18f}
and the result in Lemma~\ref{l:uQu1}
yield
\begin{align}
\abs{\nu_{\vv_i}'(w)}
&
\le 
C_1 
\brac{ 
1
+
3 C_1 C_2 
}
h_{\tau_i}^{-1} 
\norm{e - Q_{\De} e}{L^2(\tau_i)}
\notag
\\
&
\le 
C_1 
\brac{ 
1
+
3 C_1 C_2 
}
\al_5 h_{\tau_i}^{-1/2} 
\norm{e_{}}{H^{1/2}(\Om_{\tau_i})}'
\notag
\\
&
= 
C_3 h_{\tau_i}^{-1/2} 
\norm{e_{}}{H^{1/2}(\Om_{\tau_i})}',
\label{HyMul 3c}
\end{align}
where $C_3 = C_1 
\brac{ 
1
+
3 C_1 C_2 
}
\al_5$.
It follows from~\eqref{cNe e sum}, the triangle inequality, \eqref{HyMul 3c} 
and the Cauchy--Schwarz 
inequality that
\begin{align}
\inp{\cN^* e_{}}{e_{}}
&
\le 
\sum_{\vv_i\in V_{\De'}{\setminus}V_{\De}}
\abs{\nu_{\vv_i}'(w_{})}
\abs{\inpro{\cR(u_{\De})}{B_{\vv_i}'}}
\notag
\\
&
\le 
\sum_{\vv_i\in V_{\De'}{\setminus}V_{\De}}
C_3
h_{\tau_i}^{-1/2} 
\norm{e_{}}{H^{1/2}(\Om_{\tau_i})}'
\abs{\inpro{\cR(u_{\De})}{B_{\vv_i}'}}
\notag
\\
&
\le 
C_3
\brac{
\sum_{\vv_i\in V_{\De'}{\setminus}V_{\De}}
\hspace{-0.3cm}
\norm{e_{}}{H^{1/2}(\Om_{\tau_i})}'^2
}^{1/2}
\brac{
\sum_{\vv_i\in V_{\De'}{\setminus}V_{\De}}
\hspace{-0.3cm}
h_{\tau_i}^{-1}
\abs{\inpro{\cR(u_{\De})}{B_{\vv_i}'}}^2
}^{1/2}.
\label{mul Nek ek}
\end{align}
Applying Lemma~\ref{BiX lem}, we obtain 
\begin{align}
\sum_{\vv_i\in V_{\De'}{\setminus}V_{\De}}
h_{\tau_i}^{-1}
\abs{\inpro{\cR(u_{\De})}{B_{\vv_i}'}}^2
&
\le
\al_9^2
\sum_{\vv_i\in V_{\De'}{\setminus}V_{\De}}
\frac{\abs{\inpro{\cR(u_{\De})}{B_{\vv_i}'}}^2}{\norm{B_{\vv_i}'}{H^{1/2}(\mS)}
^2} .
\label{mul Nek 2}
\end{align}
We note that each  $\tau_i$ can be chosen by at most three vertices (its 
vertices). 
Therefore, we have 
\begin{align}
\sum_{\vv_i\in V_{\De'}{\setminus}V_{\De}}
\norm{e_{}}{H^{1/2}(\Om_{\tau_i})}'^2
\le 
3
\sum_{\tau_i\in \De'}
\norm{e_{}}{H^{1/2}(\Om_{\tau_i})}'^2.
\label{vi k1 tau i}
\end{align}
By applying the result in Lemma~\ref{l:coloring}, we obtain 
\begin{align}
\sum_{\tau_i\in \De_{}'}
\norm{e_{}}{H^{1/2}(\Om_{\tau_i})}'^2
\le 
\al_7 
\norm{e}{H^{1/2}(\mS)}^2.
\label{mul taui Mek}
\end{align}
It follows from~\eqref{mul Nek ek}--\eqref{mul taui Mek}  
that 
\begin{align}
\inpro{\cN^* e_{}}{e_{}}
&
\le 
C_3 (3 \al_7)^{1/2} \al_9 
% \ga_2^{-1}
\norm{e_{}}{H^{1/2}(\mS)}
\brac{
\sum_{\vv_i\in V_{\De'}{\setminus}V_{\De}}
\frac{\abs{\inpro{\cR(u_{\De})}{B_{\vv_i}'}}^2}{\norm{B_{\vv_i}'}{H^{1/2}(\mS)}
^2}
}^{1/2}.
\notag
\end{align}
This together with~\eqref{auv def}, \eqref{equ:bi li nor} and~\eqref{mu kp1 
def} 
yields 
\begin{align}
\norm{e_{}}{H^{1/2}(\mS)}
&
\le 
\al_1^{-1} 
C_3 (3 \al_7)^{1/2} \al_9  
\brac{
\sum_{\vv_i\in V_{\De'}{\setminus}V_{\De}}
\brac{\mu_{\vv_i}}^2
}^{1/2}.
\notag
\end{align}
Noting that $e = u_{\De'} - u_{\De}$ and~\eqref{HyMul 2d} we obtain 
\begin{align}
\norm{u - u_{\De}}{H^{1/2}(\mS)}
&
\le 
(1-\eta)^{-1} 
\al_1^{-1} 
C_3 (3 \al_7)^{1/2} \al_9  
\brac{
\sum_{\vv_i\in V_{\De'}{\setminus}V_{\De}}
\brac{\mu_{\vv_i}}^2
}^{1/2}.
\notag
\end{align}
The desired inequality~\eqref{HyMul 2a}
can then be obtained by denoting
\[
\al_{10}
= 
\brac{
(1-\eta)^{-1} 
\al_1^{-1} 
C_3 (3 \al_7)^{1/2} \al_9  
}^{1/2},
\]
completing the proof of the theorem.
\end{proof}

In Theorem~\ref{t:main satur}, the error 
$\norm{u-u_{\De}}{H^{1/2}(\mS)}$ is bounded above by the 
sum of nodal
estimators. For refinement purpose, the a posteriori error estimate can also be 
written
in the the form of element estimators as in the following corollary.

\begin{corollary}\label{c:Hier est}
Let all assumptions in Theorem~\ref{t:main satur} be satisfied. Then there holds
\begin{align}
\norm{u-u_{\De}}{H^{1/2}(\mS)}^2 
\le 
\al_{10}
\sum_{\tau\in \De}
\theta_{\De}(\tau)^2,
\label{hier est tau}
\end{align}
where 
\begin{equation}\label{local err est num}
\theta_{\De}(\tau)^2
=
\sum_{\vv\in V_{\De'}{\setminus}V_{\De}
\atop\vv\in \tau}
\mu_{\vv}^2.
\end{equation}
\end{corollary}

\section{Mesh Refinement}\label{s:mesh ref}
In this section, 
we briefly discuss the mesh refinement technique that will be used to refine 
our spherical triangulations. The technique is based on the a posteriori error 
estimates proved in 
Theorems~\ref{t:upp bound} and~\ref{t:main satur}, and Corollary~\ref{c:Hier 
est}. 
Borrowing existing ideas in planar 
cases, see e.g. 
\cite{Bnsch1991,Binev2004,Carstensen2004,CASCON2007,Nochetto2009,Rivara1984-3,
Verfrth1994}, our
mesh refinement algorithms 
consist of two subroutines. One is 
constructing the indicators from the 
error estimators. The other is defining the 
rules that are used to divide 
the triangles.
Here, 
indicator constructions are different for the two adaptive approaches which 
are based on the residual and the hierarchical estimates. Meanwhile,
we use the same rule to divide the triangles for  both adaptive procedures. 

\vspace{0.2cm}

\textbf{Residual adaptive approach:}
Starting with a spherical triangulation $\De_k$,
we denote by $\wth{\De}_k$ the subset of $\De_k$  
containing all spherical triangles that will be refined.
This can be achieved 
with the following marking strategy (see \cite{Drfler1996}):

\vspace{0.2cm}

\noindent\textbf{Strategy:}
\textit{Given a parameter $0<\xi<1$, construct 
a minimal subset $\wth\De_k$ of $\De_k$ such 
that
\begin{align}
\sum_{\tau \in \wth\De_k} 
\eta_{\De_k,1/2}(\tau)^2
\ge
\xi^2\, 
\sum_{\tau \in \De_k}
\eta_{\De_k,1/2}(\tau)^2,
\notag
\end{align}
and mark all spherical triangles in 
$\wth\De_k$ for refinement. 
Here, recall that $\eta_{\De_k,1/2}(\tau)$ is defined by~\eqref{etaXtau}.
}

\vspace{0.2cm}

\textbf{Hierarchical adaptive approach:}
Starting with a spherical triangulation $\De_k$,
we denote by $\De_k'$ the finer mesh of $\De_k$ which is created by joining the 
midpoints of the three edges of all triangles in $\De_k$, see 
Figure~\ref{f:divide 2}. 
Note here that we only need 
the vertices of $\De_k'$ in order to compute the nodal estimators
\[
\mu_{\vv}, 
\quad 
\vv\in V_{\De_k'},
\]
see~\eqref{mu kp1 def}.
The mesh $\De_k'$ is not at all the finer mesh that we use to create 
approximation spaces.
For each $\tau$ in $\De_k$, the local error estimator is computed by 
\begin{equation*}
\theta_{\De_k}(\tau)^2
=
\sum_{\vv\in V_{\De_k'}{\setminus}V_{\De_k}
\atop\vv\in \tau}
\mu_{\vv}^2,
\end{equation*}
see~\eqref{local err est num}.
The subset  $\wth \De_k$ of spherical triangles in $\De_k$ which will be 
marked for refinement is determined by applying the above strategy:

\vspace{0.2cm}

\textit{Given a parameter $0<\xi<1$, construct 
a minimal subset $\wth \De_k$ of $\De_k$ such 
that
\begin{align}
\sum_{\tau \in \wth\De_k} 
\theta_{\De_k}(\tau)^2
\ge
\xi^2\, 
\sum_{\tau \in \De_k}
\theta_{\De_k}(\tau)^2,
\notag
\end{align}
and mark all spherical triangles in 
$\wth \De_k$ for refinement.
}

\vspace{0.2cm}

Once, the subset $\wth\De_k$ of spherical triangles in $\De_k$ that are to be 
divided is obtained, 
mesh refinement techniques are then applied.
When it comes to the mesh refinement, algorithms 
for cutting triangles in  triangulations have 
been extensively discussed 
in~\cite{Rivara1984-3}. These algorithms are 
based 
on the 
bisection of triangles by dividing the longest 
edges so that the following features are 
satisfied.
Let $\De_k$ be a conforming triangulation, 
i.e. the 
intersection of two non-disjoint, nonidentical 
triangles is either a common vertex or common 
edge.
With any refinement submesh $\wth{\De}_k\in 
\De_k$, the algorithm produces a new conforming 
triangulation $\De_{k+1}$ with the following 
properties:

\begin{enumerate}[(i)]
\item  
all elements of $\wth{\De}_k$ 
are refined to create new elements in 
$\De_{k+1}$,
\item
$\De_{k+1}$ is nested in 
$\De_k$ in such a way that each refined triangle 
is 
embedded in one triangle of $\De_k$,
\item
$\De_{k+1}$ is 
non-degenerated, i.e. the interior angles of all 
triangles of $\De_{k+1}$ are guaranteed to be 
bounded away from 0,
\item
the transition between large 
and small triangles is not abrupt.
\end{enumerate}

Following~\cite{Verfrth1994}, 
the below steps are used to 
 produce a 
totally refined and conforming triangulation 
$\De_{k+1}$ in the following way:

\begin{itemize}
\item [Step 1:] 
Separate all $\tau$ in 
$\wth{\De}_k$ into 4 pieces to obtain $\wtd{\De_k}$,
see Figure~\ref{triangle11}(a).
\item[Step 2:] 
Find all hanging nodes in 
$\wtd{\De_k}$ and verify if 
each of these hanging nodes lies 
on the longest edge of  a triangle or not. 
\begin{itemize}
\item 
If the hanging node 
lies on the longest edge, join it with the 
opposite vertex to obtain 2 new triangles, see 
Figure~\ref{triangle11}(b).
\item 
If the hanging node 
does not lie on the longest edge, join 
it with the middle point of the longest edge, 
together with joining the middle point of the 
longest edge with its opposite vertex to obtain 3 
new triangles, see Figure~\ref{triangle11}(c).
\end{itemize}
\end{itemize}
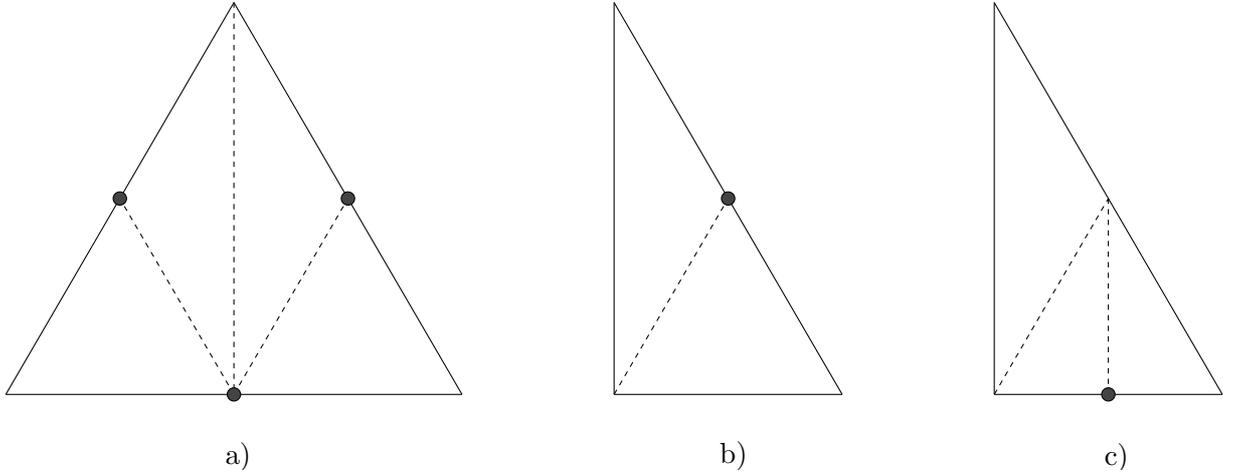
\begin{figure}
	\centering
\definecolor{uuuuuu}{rgb}{0.26666666666666666,
0.26666666666666666,0.26666666666666666}
\begin{tikzpicture}
% [line cap=round,line 
% join=round,>=triangle 45,x=0.75cm,y=0.75cm]
\clip(0.6,-1.) rectangle (17.1,5.6);
\draw (4.,5.196152422706632)-- (1.,0.);
\draw (4.,5.196152422706632)-- (7.,0.);
\draw (7.,0.)-- (1.,0.);
\draw [dash pattern=on 2pt off 2pt] 
(4.,5.196152422706632)-- (4.,0.);
\draw (9.,0.)-- (9.,5.196152422706631);
\draw (9.,5.196152422706631)-- (12.,0.);
\draw (12.,0.)-- (9.,0.);
\draw [dash pattern=on 2pt off 2pt] (9.,0.)-- 
(10.5,2.5980762113533156);
\draw (14.,5.19615242270663)-- (17.,0.);
\draw (14.,0.)-- (17.,0.);
\draw [dash pattern=on 2pt off 2pt] (14.,0.)-- 
(15.5,2.598076211353315);
\draw [dash pattern=on 2pt off 2pt] 
(15.5,2.598076211353315)-- (15.5,0.);
\draw (14.,5.19615242270663)-- (14.,0.);
\draw [dash pattern=on 2pt off 2pt] (4.,0.)-- 
(5.5,2.598076211353316);
\draw [dash pattern=on 2pt off 2pt] (4.,0.)-- 
(2.5,2.598076211353316);
\draw (3.751727854150801,-0.5047814587725695) 
node[anchor=north west] {a)};
\draw (10.252109736068247,-0.4802053457974563) 
node[anchor=north west] {b)};
\draw (15.314789008941569,-0.5170695152601261) 
node[anchor=north west] {c)};
\begin{scriptsize}
\draw [fill=uuuuuu] (4.,0.) circle (2.5pt);
\draw [fill=uuuuuu] (10.5,2.5980762113533156) 
circle (2.5pt);
\draw [fill=uuuuuu] (15.5,0.) circle (2.5pt);
\draw [fill=uuuuuu] (5.5,2.598076211353316) 
circle 
(2.5pt);
\draw [fill=uuuuuu] (2.5,2.598076211353316) 
circle 
(2.5pt);
\end{scriptsize}
\end{tikzpicture}
\caption{Possible cases of refined 
triangles}\label{triangle11}
\end{figure}

\section{Numerical Experiments}
\label{s:num ex}
We consider the exterior Neumann problem

\begin{equation}\label{Neu prob num exp}
\begin{aligned}
\De U(\vecx)
&
= 0 
\quad 
\text{for all } 
\abs{\vecx}>1,
\\
\frac{\partial U(\vecx)}{\partial\nu}
&
=
Z_N(\vecx)
\quad 
\text{for all }
\vecx\in \mS,
\\
U(\vecx)
&
=
\cO\brac{\abs{\vecx}^{-1}}
\quad 
\text{when }
\abs{\vecx}\goto \infty,
\end{aligned}
\end{equation}
 where the boundary data $Z_N$ is one of the following functions
\begin{align}\label{zn ex}
Z_1(\vx)
=
\frac{\vp \cdot \vx-1}{\snorm{\vx-\vp}{}^3}
-
1
% \quad
% \text{with }
% \vp=(0,0,0.95),\,
\end{align}
and 
\begin{align}\label{zn ex2}
Z_2(\vx)
=
\frac{\vp \cdot \vx-1}{\snorm{\vx-\vp}{}^3}
-
\frac{\vq \cdot \vx-1}{\snorm{\vx-\vq}{}^3},
\end{align}
where 
$\vp=(0,0,0.95)$ and $\vq = (0,0,-0.95)$. Solving the problem~\eqref{Neu prob 
num exp} is equivalent to solving the hypersingular integral equation
\begin{align}\label{hyper ex}
-Nu+\int_{\mS}u\,d\sigma = f
\quad
\text{on}
\quad
\mS,
\end{align}
see e.g.~\cite{Steinbach08, TraLeGSloSte09a}.
Here, the right hand side $f$ of~\eqref{hyper 
ex} is given by
\begin{align}\label{f = zn_1}
f_k(\vx)
=
\frac{1}{2} Z_k(\vx)
+D^* Z_k(\vx), 
\quad 
\vx\in \mS,
\end{align}
for $k=1,2$, 
and
the operator $D^*$ is defined by 
\[
D^* v(\vecx)
=
\int_{\mS}
\frac{\partial}{\partial\nu_{\vecx}}
\frac{1}{\abs{\vecx-\vecy}}
v(\vecy)\,d\sigma_{\vecy},
\quad 
\vecx\in \mS,
\]
see~\cite[page 122]{Ned00}.
The exact solution of the exterior Neumann problem~\eqref{Neu prob num exp} is
\begin{align}\label{exact U}
U_1(\vx) 
=
\frac{1}{\snorm{\vx-\vp}{}}
-
\frac{1}{\abs{\vx}}
\quad 
\text{and}
\quad
U_2(\vx)
=
\frac{1}{\snorm{\vx-\vp}{}}
-
\frac{1}{\snorm{\vx-\vq}{}},
\quad 
\abs{\vx} > 1.
\notag
\end{align}
and the exact solution to the hypersingular integral equation~\eqref{hyper ex} 
is given by
\begin{align}
u_1(\vecx)
=
\frac{1}{\snorm{\vx-\vp}{}}
-
1
\quad
\text{and}
\quad 
u_2(\vx) 
=
\frac{1}{\snorm{\vx-\vp}{}}
-
\frac{1}{\snorm{\vx-\vq}{}},
\quad 
\vx\in\mS.
\end{align}

We solve~\eqref{hyper ex} 
by 
using the Galerkin method with 
$S(\De)$, the space of continuous {piecewise} linear spherical splines.
Here, {the spherical triangulations $\De$ are 
obtained in three different ways: uniform,
residual and hierarchical
adaptive mesh refinements.}
For experimental purposes, we 
start with an initial triangulation of 
eight equal spherical triangles with six 
nodes (two at the poles and four
on the equator). 
For the uniform meshes,
every further 
refinement consists of partitioning every
spherical triangle into four smaller spherical 
triangles by joining the midpoints of the edges, {see Figure~\ref{f:divide 
2}.} 
This 
guarantees that all triangles in the spherical 
triangulations obtained after refinements are of
a finite 
number of similarly distinct triangles.
{For the residual and hierarchical adaptive meshes, we apply the 
strategies in Section~\ref{s:mesh ref} to refine 
the 
meshes 
after estimating the element  errors, 
$\eta_{\De,1/2}(\tau)$ and $\theta_{\De}(\tau)$, 
see~\eqref{etaXtau} 
and~\eqref{local err est num}, respectively. 
}

\vspace{0.5cm}

\begin{figure}[h]
\begin{center}
\begin{tikzpicture}[xscale=1.5, yscale=1.5]

%%%%%%%%%%%%%%%%%%%%%
\draw (0,0)--(3,0)--(1.5,2.6)--(0,0);

\draw [fill](2.25,1.3) circle (0.05cm);
\draw [fill](0.75,1.3) circle (0.05cm);
\draw [fill](1.5,0) circle (0.05cm);

\draw[->] (3,1.3)--(4,1.3);
%%%%%%%%%%%%%%%%%%%%%
\draw (4,0)--(7,0)--(1.5+4,2.6)--(4,0);

\draw (2.25+4,1.3)--(0.75+4,1.3)--(1.5+4,0)--(2.25+4,1.3);

\draw [fill](2.25+4,1.3) circle (0.05cm);
\draw [fill](0.75+4,1.3) circle (0.05cm);
\draw [fill](1.5+4,0) circle (0.05cm);
%%%%%%%%%%%%%%%%%%%%%

\end{tikzpicture}
\end{center}
\caption{Uniform mesh refinement}
\label{f:divide 2}
\end{figure}
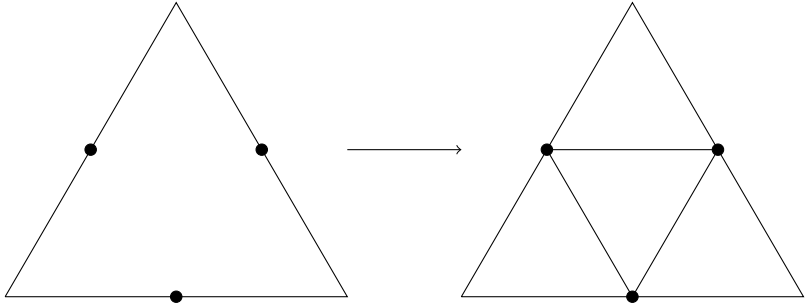

\vspace{0.5cm}

Suppose that 
$V_{\De} = \sett{\vecv_1, \ldots, \vecv_M}$ is the set of all vertices in the 
spherical triangulation~$\De$. 
We choose a basis for $S(\De)$ to be the set 
\[
\sett{B_{\vecv_i}: i = 1,\ldots, M},
\]
where $B_{\vecv_i}$ is the basis function associated with the vertex $\vecv_i$, 
see~\eqref{Bvi def}. 
We denote by $u_{\De}\in S(\De)$ the Galerkin solution to~\eqref{hyper ex}. 
Then 
$u_{\De}  = \sum_{i=1}^M \nu_i B_{\vecv_i}$,
where $\nu_i\in \R$ for $i=1,\ldots,M$, satisfies
\[
a(u_{\De}, B_{\vecv_j}) 
=
\inpro{f}{B_{\vecv_j}},
\quad 
j=1,\ldots, M.
\]
This results in the following  matrix equation 
\begin{equation}\label{mat equ}
\bsb{A} \bsb{\nu} = \bsb{F}.
\end{equation}
The entry $A_{ij}$, for $i,j = 1,\ldots, M$, of 
the stiffness 
matrix $\bsb A$ is computed by
\begin{align}
A_{ij}
&=
-\frac{1}{4 \pi}
\int_{\mS}(N B_{\vv_i})(\vx) B_{\vv_j}(\vx)
d\sigma_{\vx}
d\sigma_{\vy}
+
\int_{\mS}B_{\vv_i}(\vx)
d\sigma_{\vx}
\int_{\mS} B_{\vv_j}(\vy)
d\sigma_{\vy}.
\label{Aij hyp ex}
\end{align}
The first integral in~\eqref{Aij hyp ex} is computed by 
\begin{equation}\label{equ:dou int stif}
\begin{aligned}
-\int_{\mS} (N B_{\vecv_i})(\vecx) B_{\vecv_j}(\vecx) \,d\sigma_{\vecx} & = 
\frac{1}{4\pi} \, \int_{\mS}
\int_{\mS} \frac{\overrightarrow{\curl}_{\mS} B_{\vecv_i}(\vecx)\cdot
\overrightarrow{\curl}_{\mS} B_{\vecv_j}(\vecy)}{\snorm{\vecx -\vecy}{}}
\,d\sigma_{\vecx}\,d\sigma_{\vecy}
\\
&
=
\frac{1}{4\pi}
\sum_{\tau\in\Delta}
\sum_{\tau'\in\Delta}
\int_{\tau}
\int_{\tau'}
\frac{\overrightarrow{\curl}_{\mS} B_{\vecv_i}(\vecx)\cdot
\overrightarrow{\curl}_{\mS} B_{\vecv_j}(\vecy)}{\snorm{\vecx -\vecy}{×}}
\,d\sigma_{\vecx}\,d\sigma_{\vecy},
\end{aligned}
\end{equation}
see~\cite[Theorem 
3.3.2]{Ned00}.
Here, $\overrightarrow{\rm{curl}}_\mS v$ is the vectorial
surface
rotation defined by
\[
\overrightarrow{\rm{curl}}_\mS v
=
-
\frac{\partial v}{\partial \theta}
\overrightarrow{e_\varphi}
+
\frac{1}{\sin\theta}
\frac{\partial v}{\partial \varphi}
\overrightarrow{e_\theta},
\]
where $\overrightarrow{e_\varphi}$, $\overrightarrow{e_\theta}$
are the two unit vectors corresponding to the Euler angles.
{
Computation of the double integrals in~\eqref{equ:dou int stif}
requires evaluation of integrals of the type
\begin{equation}\label{iint tau1 tau2}
\int_{\tau^{(1)}}
\int_{\tau^{(2)}}
\frac{f_1(\vecx)\, f_2(\vecy)}{\snorm{\vecx - \vecy}{}}
\,d\sigma_{\vecx}
d\sigma_{\vecy},
\end{equation}
where $\tau^{(1)}$ and $\tau^{(2)}$ are spherical triangles
in $\De$ and the functions $f_1$ and $f_2$ are analytic
for all $\vecx\in \tau^{(1)}$ and $\vecy\in \tau^{(2)}$.
For more details about the above evaluation, please refer to
\cite{PhamTran12,PhamTranChernov11}.
}

The right hand side $\bsb F$ of 
the linear system~\eqref{mat equ} has entries given by
\begin{align*}
F_i
&
=
\int_{\mS} 
B_{\vecv_i}(\vx)\,f(\vx) d\sigma_{\vx} 
=
\frac{1}{2}
\int_{\mS} 
B_{\vecv_i}(\vx)\,Z_N(\vx) d\sigma_{\vx} 
+
\frac{1}{2}
\int_{\mS} 
B_{\vecv_i}(\vx)\,(D^*Z_N)(\vx) d\sigma_{\vx},
\end{align*}
for all $i = 1,\ldots, M$.
Once solving the matrix equation~\eqref{mat equ}, 
we obtain the coefficient vector $\bsb\nu = 
(\nu_1,\ldots, \nu_M)$ and thus the approximate 
solution $u_{\De} = \sum_{i=1}^M \nu_i B_{\vecv_i}$. The 
error $\norm{u-u_{\De}}{H^{1/2}(\mS)}$ is then computed 
by 
\begin{align}
\norm{u-u_{\De}}{H^{1/2}(\mS)}^2 
&
\simeq 
a(u-u_{\De}, u-u_{\De}) 
=
a(u-u_{\De}, u) 
\notag
\\
&
=
a(u,u) - a(u,u_{\De}) 
=
\inpro{f}{u}
-
\inpro{f}{u_{\De}},
\notag
\end{align}
noting~\eqref{auv cont}--\eqref{GalSol1}.

{
\begin{table}[h]
\centering
\caption{Errors vs degrees of freedom  for $f_1$}
\label{tabEx1}

\vspace{0.2cm}

\begin{tabular}{|c|c|c|c|c|c|}
\hline 
\multicolumn{2}{|c|}{} 
& 
\multicolumn{2}{|c|}{}
& 
\multicolumn{2}{|c|}{} 
\\[-1em]
\multicolumn{2}{|c|}{Uniform} 
& 
\multicolumn{2}{|c|}{Residual} 
&
\multicolumn{2}{|c|}{hierarchical} 
\\[0.1em]
\hline
\rule{0pt}{15pt}
DoFs &    Error &    
DoFs &    Error
&    DoFs &       Error \\
\hline
\rule{0pt}{15pt}
6   & 0.77566 
&
6   & 0.77566 
&
6  & 0.77566 
\\
18  & 0.38229 
&
26  & 0.43544
&
14  & 0.68900 
\\
66  & 0.16686 
&
78  & 0.07714
&
95  & 0.18822 
\\
258 & 0.09537 
&
102 & 0.04493
&
119  & 0.07424 
\\
1026& 0.05792 
&
128 & 0.03864
&
141  & 0.04222 
\\
4098& 0.03564  
&
211 & 0.03495 
&
170 & 0.03574  
\\ 
\hline
\end{tabular}
\end{table}
}

{
\begin{table}[h]
\centering
\caption{Degrees of freedom and accumulating computation time  for $f_1$}
\label{tabEx1a}

\vspace{0.2cm}

\begin{tabular}{|c|c|c|c|c|c|}
\hline 
\multicolumn{2}{|c|}{} 
& 
\multicolumn{2}{|c|}{}
& 
\multicolumn{2}{|c|}{} 
\\[-1em]
\multicolumn{2}{|c|}{Uniform} 
& 
\multicolumn{2}{|c|}{Residual} 
&
\multicolumn{2}{|c|}{hierarchical} 
\\[0.1em]
\hline
\rule{0pt}{15pt}
DoFs &    Comp. time 
&    
DoFs &    Comp. time
&    
DoFs &    Comp. time
\\
\hline
\rule{0pt}{15pt}
6   & 1.58  
&
6   & 1.58  
&
6  & 2.54  
\\
18  & 7.09 
&
26  & 11.07 
&
14  & 9.60 
\\
66  & 30.12  
&
78  & 53.41 
&
95  & 125.18 
\\
258 & 192.91 
&
102 & 91.39
&
119  & 245.08 
\\
1026 & 2654.11 
&
128 & 144.25 
&
141  & 401.22 
\\
4098& 38754.89 %28650*2654/1962%  
&
211 & 259.89 
&
170 & 612.70 
\\ 
\hline
\end{tabular}
\end{table}
}

\begin{table}[h]
\centering
\caption{Errors vs degrees of freedom  for $f_2$}
\label{tabEx2}

\vspace{0.2cm}

\begin{tabular}{|c|c|c|c|c|c|}
\hline 
\multicolumn{2}{|c|}{} 
& 
\multicolumn{2}{|c|}{}
& 
\multicolumn{2}{|c|}{} 
\\[-1em]
\multicolumn{2}{|c|}{Uniform} 
& 
\multicolumn{2}{|c|}{Residual} 
&
\multicolumn{2}{|c|}{hierarchical} 
\\[0.1em]
\hline
\rule{0pt}{15pt}
DoFs &    Error &    
DoFs &    Error
&    DoFs &       Error \\
\hline
\rule{0pt}{15pt}
6   & 0.78050  
&
6   & 0.78050  
&
6   & 0.78050  
\\
18  & 0.36153 
&
40 & 0.38340 
&
54 & 0.38262  
\\
66  & 0.15705
&
151 & 0.06762 
&
153 & 0.16873 
\\
258 & 0.09356 
&
199 & 0.04232 
&
199 & 0.06693 
\\
1026& 0.05826 
&
253 & 0.03668 
&
247 & 0.04151
\\
4098& 0.03682  
&
448 & 0.03269
&
302 & 0.03606
\\ 
\hline
\end{tabular}
\end{table}

\begin{table}[h]
\centering
\caption{Degrees of freedom and accumulating computation time  for $f_2$}
\label{tabEx2a}

\vspace{0.2cm}

\begin{tabular}{|c|c|c|c|c|c|}
\hline 
\multicolumn{2}{|c|}{} 
& 
\multicolumn{2}{|c|}{}
& 
\multicolumn{2}{|c|}{} 
\\[-1em]
\multicolumn{2}{|c|}{Uniform} 
& 
\multicolumn{2}{|c|}{Residual} 
&
\multicolumn{2}{|c|}{hierarchical} 
\\[0.1em]
\hline
\rule{0pt}{15pt}
DoFs &    Comp. time 
&    
DoFs &    Comp. time
&    
DoFs &    Comp. time
\\
\hline
\rule{0pt}{15pt}
6   & 1.67  
&
6   & 2.01  
&
6  & 3.68  
\\
18  & 7.49 
&
40  & 27.59 
&
54  & 88.60 
\\
66  & 31.44  
&
151  & 176.24 
&
153  & 346.21 
\\
258 & 184.11 
&
199 & 311.59 
&
199  & 722.11 
\\
1026 & 2421.76 
&
253 & 509.00 
&
247  & 1242.09 
\\
4098 &  35351.71 %38754*2421/2654
&
448 & 1051.12 
&
302  & 1968.70 
\\ 
\hline
\end{tabular}
\end{table}

We solve \eqref{hyper ex} by using  uniform, residual and hierarchical
adaptive refinements for 
the 
right hand sides $f_1$ and $f_2$ being
defined by \eqref{f = zn_1}. For both examples, we 
find approximate solutions, compute
the errors, degrees of freedom and accumulating computation time, see 
Tables~\ref{tabEx1}--\ref{tabEx2a}.
{We note here that the convergence rates  of the uniform refinement method
for both $f_1$ and $f_2$ are slightly smaller than theoretical results. The 
errors behave roughly 
$\mathcal{O}(M^{-1.24/2})$ instead of $\mathcal{O}(M^{-1.5/2})$ as suggested 
by~\eqref{a priori hyper}.
This may be due to the small number 
of uniform meshes that have have been used and the low number
of elements in these meshes.
}

The numerical results suggest significant advantages of the two adaptive
refinement approaches in terms of required degrees of freedom and accumulating 
computation time, see also Figures~\ref{f1:ErrDof}--\ref{f2:ErrTime}.
For example, to obtain an accuracy of around 3.5\% when solving~\eqref{hyper 
ex} for $f_1$, while the uniform refinement 
approach requires $4098$ degrees of freedom (see 
Figure~\ref{f:uniform4098}) 
and the corresponding computation time is almost $10.7$ hours, our 
residual and hierarchical adaptive 
refinement counterparts need only $211$ and $170$ vertices 
%(see Figure~\ref{f1:adaptives}) 
and it takes only more than $10$ minutes to 
complete 
the calculation, see Tables~\ref{tabEx1}--\ref{tabEx1a} and 
Figures~\ref{f1:ErrDof}--\ref{f1:ErrTime}. 
Similar advantages of the adaptive refinement approaches are also observed when 
solving~\eqref{hyper ex} for $f_2$ given by~\eqref{f = zn_1} and~\eqref{zn 
ex2}, see Tables~\ref{tabEx2}--\ref{tabEx2a} and 
Figures~\ref{f2:ErrDofs}--\ref{f2:ErrTime}. 
For example, to obtain an accuracy of $3.6\%$, uniform refinement method has to 
use the uniform mesh of $4098$ vertices and the calculation takes nearly $10$ 
hours to complete. Meanwhile, the residual adaptive method requires a mesh of 
$448$ nodes and the (accumulating) computation time is about $17.5$ minutes. 
The numbers for the hierarchical adaptive counterpart are $302$ nodes and 
$32.8$ minutes, respectively.

Figure~\ref{f1:adaptives} shows adaptive meshes obtained when we solve the 
equation~\eqref{hyper ex} with the right hand side $f_1$ by using the residual 
and hierarchical refinement 
approaches. 
Denser areas of nodes surrounding the north pole are observed. The spherical 
triangulations shown in 
Figures~\ref{f2:Re_adaptive448} and~\ref{f2:ReAdap302}
are the $448$-node and $302$-node meshes obtained when we solve~\eqref{hyper 
ex} with the right hand side $f_2$ by using the two adaptive methods. In these 
two figures, we witness denser areas surrounding the north and south poles.
These denser areas are due to the fact that their contributions to the total
errors are higher than other regions on the unit sphere, and thus must be 
accordingly refined as discussed in Section~\ref{s:mesh ref}.

\begin{figure}[h]
\begin{center}
\caption{Errors vs DoFs for $f_1$}
\label{f1:ErrDof}
\includegraphics[width=8cm]{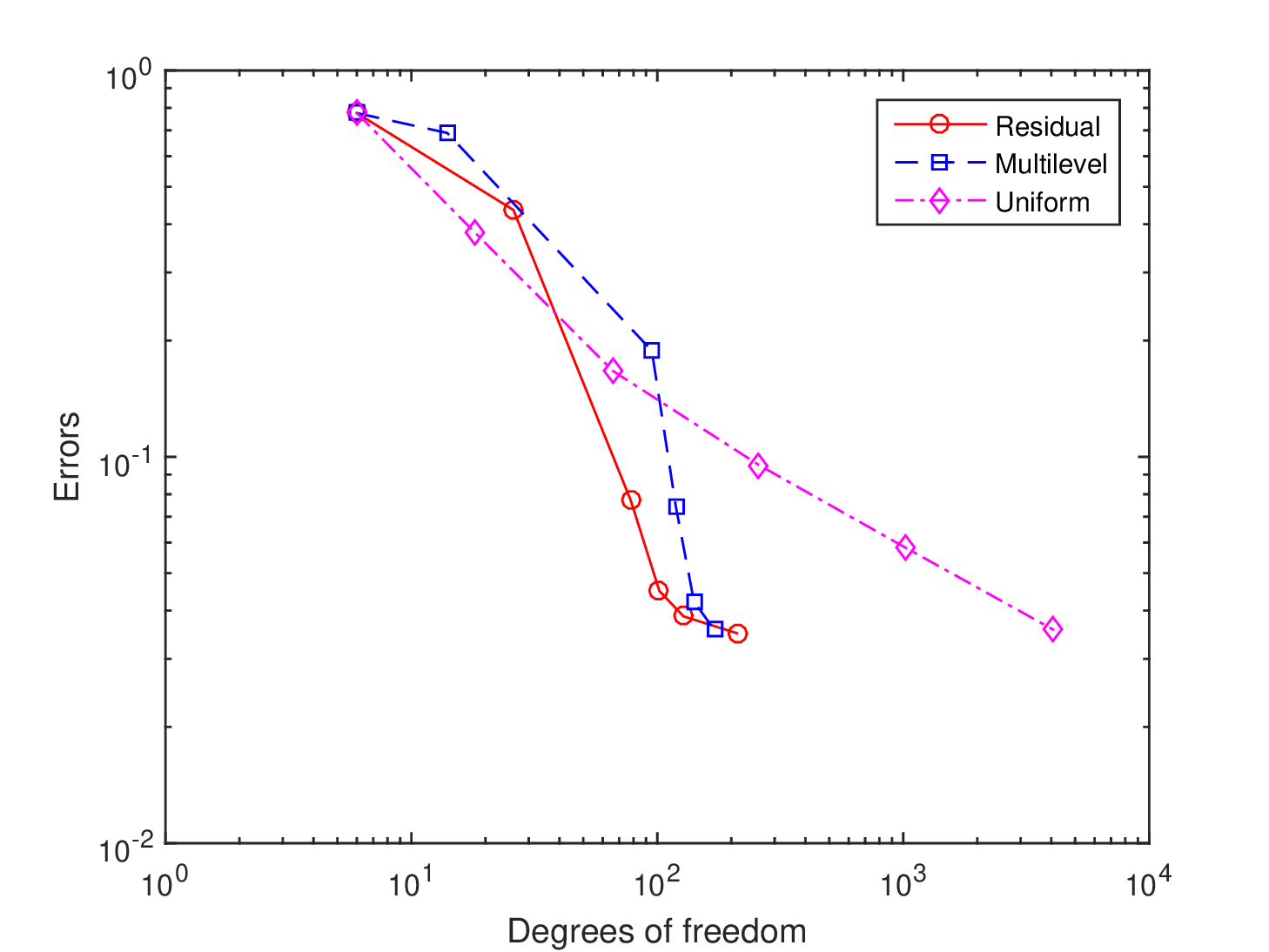}
\end{center}
\end{figure}

\begin{figure}[h]
\begin{center}
\caption{Errors vs Accumulating computation time for $f_1$}
\label{f1:ErrTime}
\includegraphics[width=8cm]{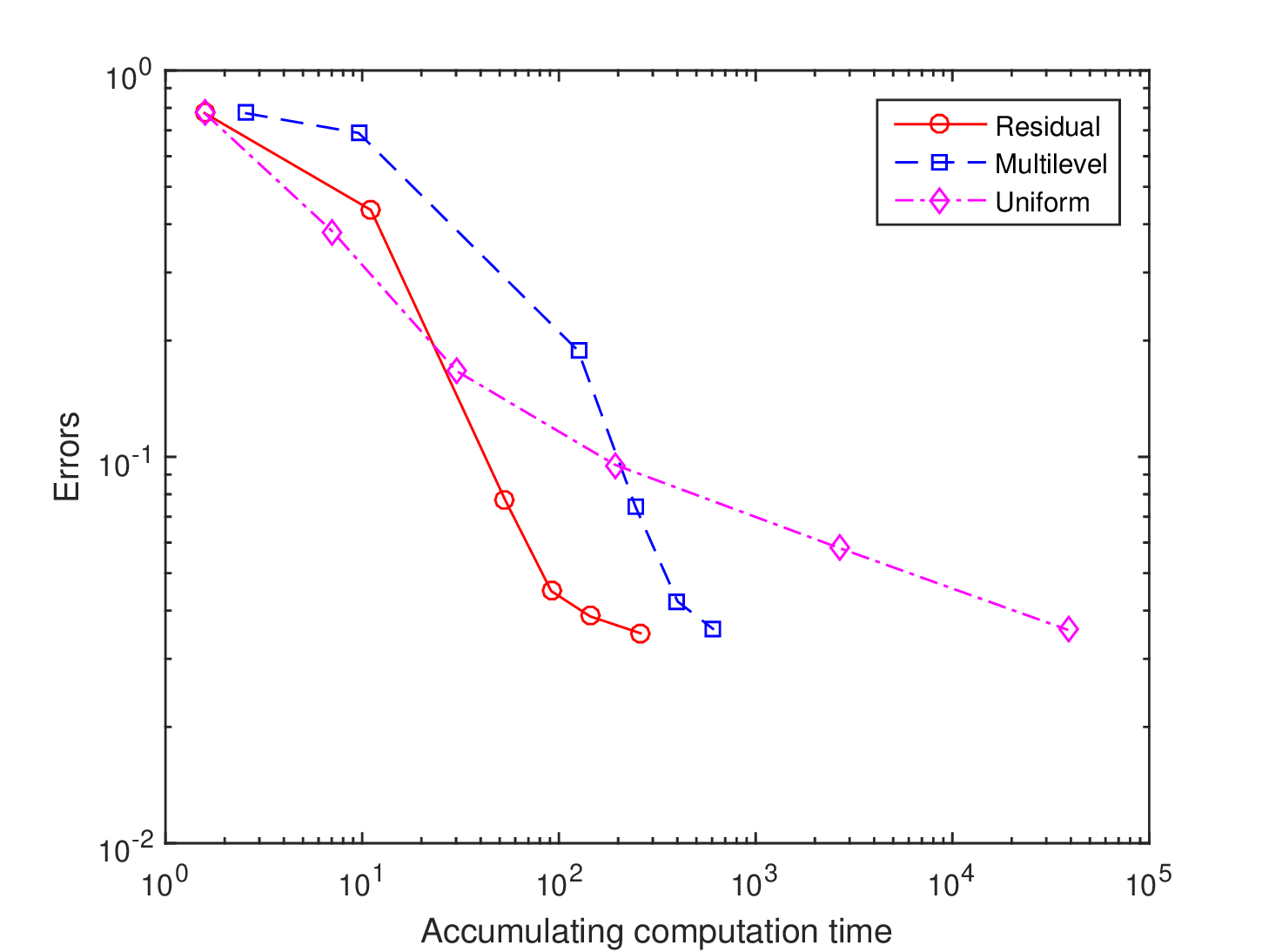}
\end{center}
\end{figure}

\begin{figure}[h]
\begin{center}
\caption{Errors vs DoFs for $f_2$}
\label{f2:ErrDofs}\includegraphics[width=8cm]{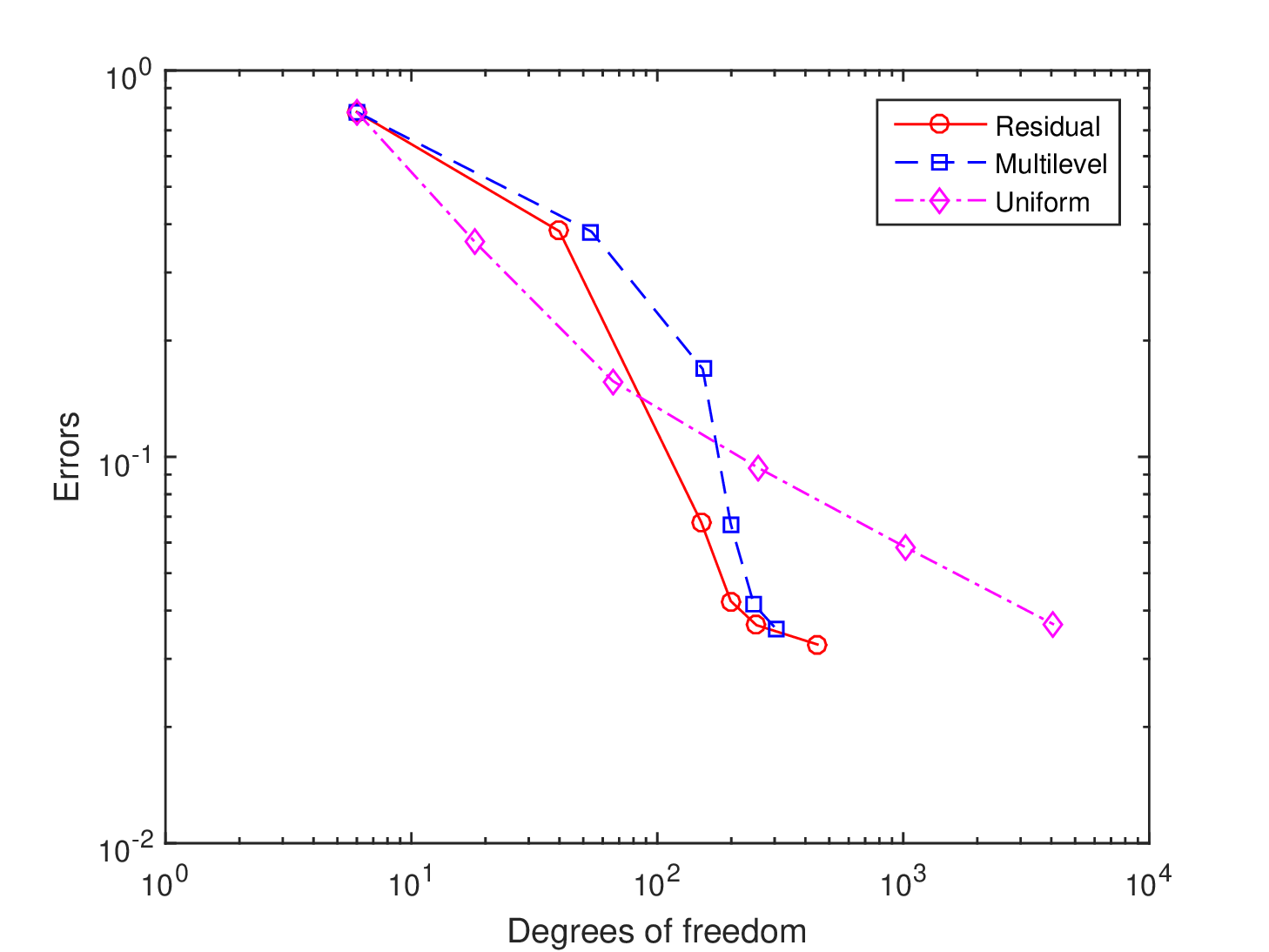}
\end{center}
\end{figure}

\begin{figure}[h]
\begin{center}
\caption{Errors vs Accumulating computation time for $f_2$}
\label{f2:ErrTime}
\includegraphics[width=8cm]{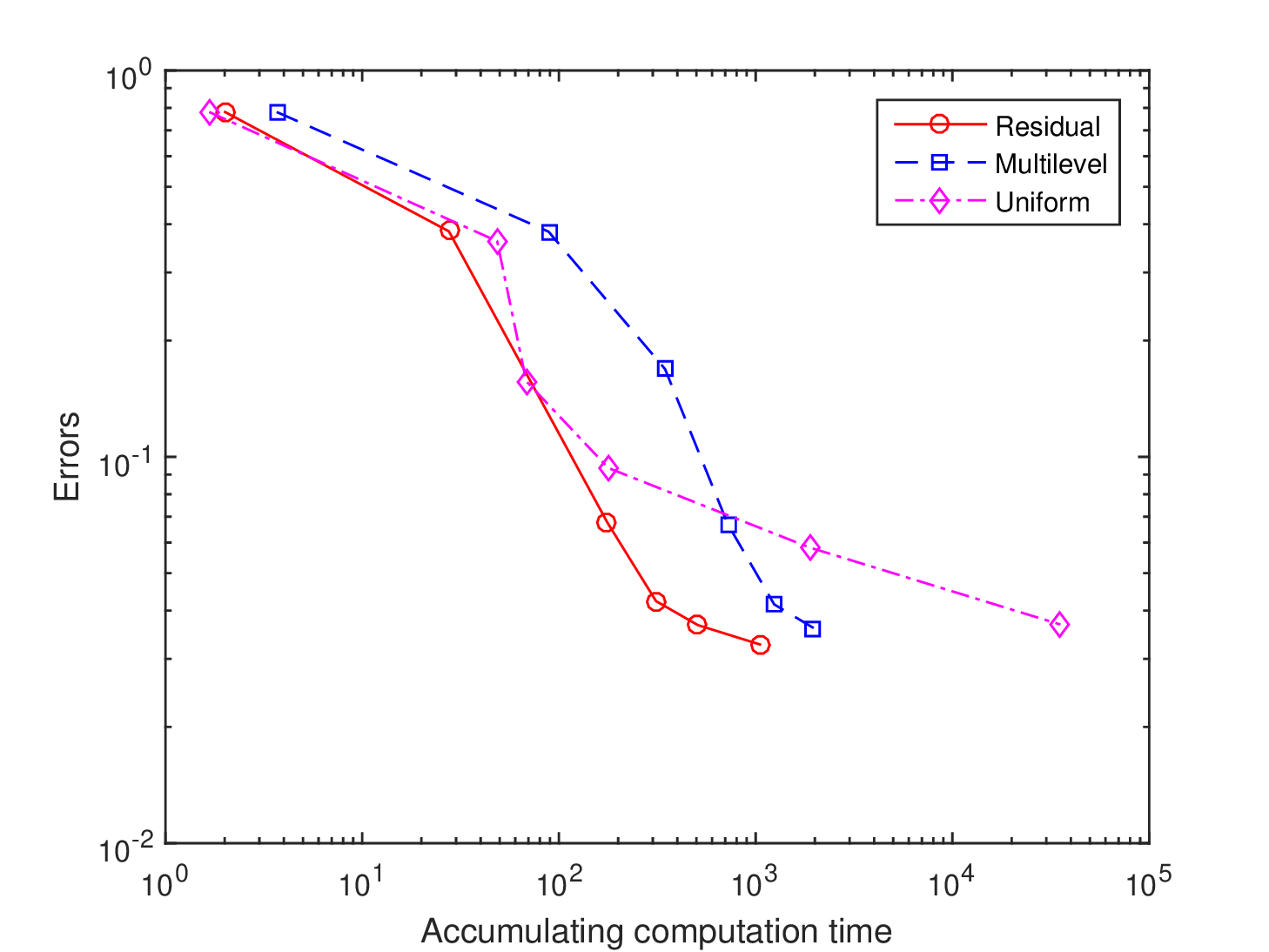}
\end{center}
\end{figure}

% \begin{comment}
\begin{figure}[h]
\begin{center}
\caption{Uniform triangulation with 4098 vertices}
\label{f:uniform4098}
\includegraphics[width=8cm]{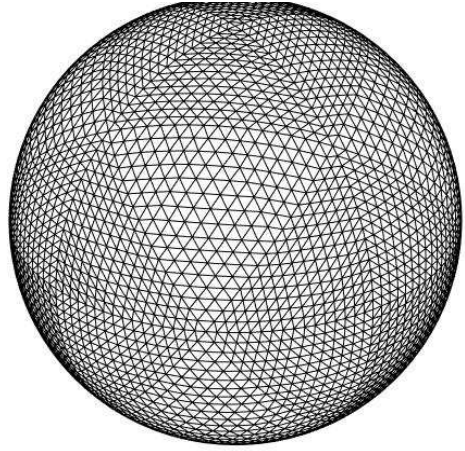}
\end{center}
\end{figure}
% \end{comment}

\begin{figure}[h]
\caption{Adaptive triangulations for $f_1$}
\label{f1:adaptives}

\vspace{0.5cm}

\centering
\begin{subfigure}{.5\textwidth}
  \centering
  \includegraphics[width=7cm]{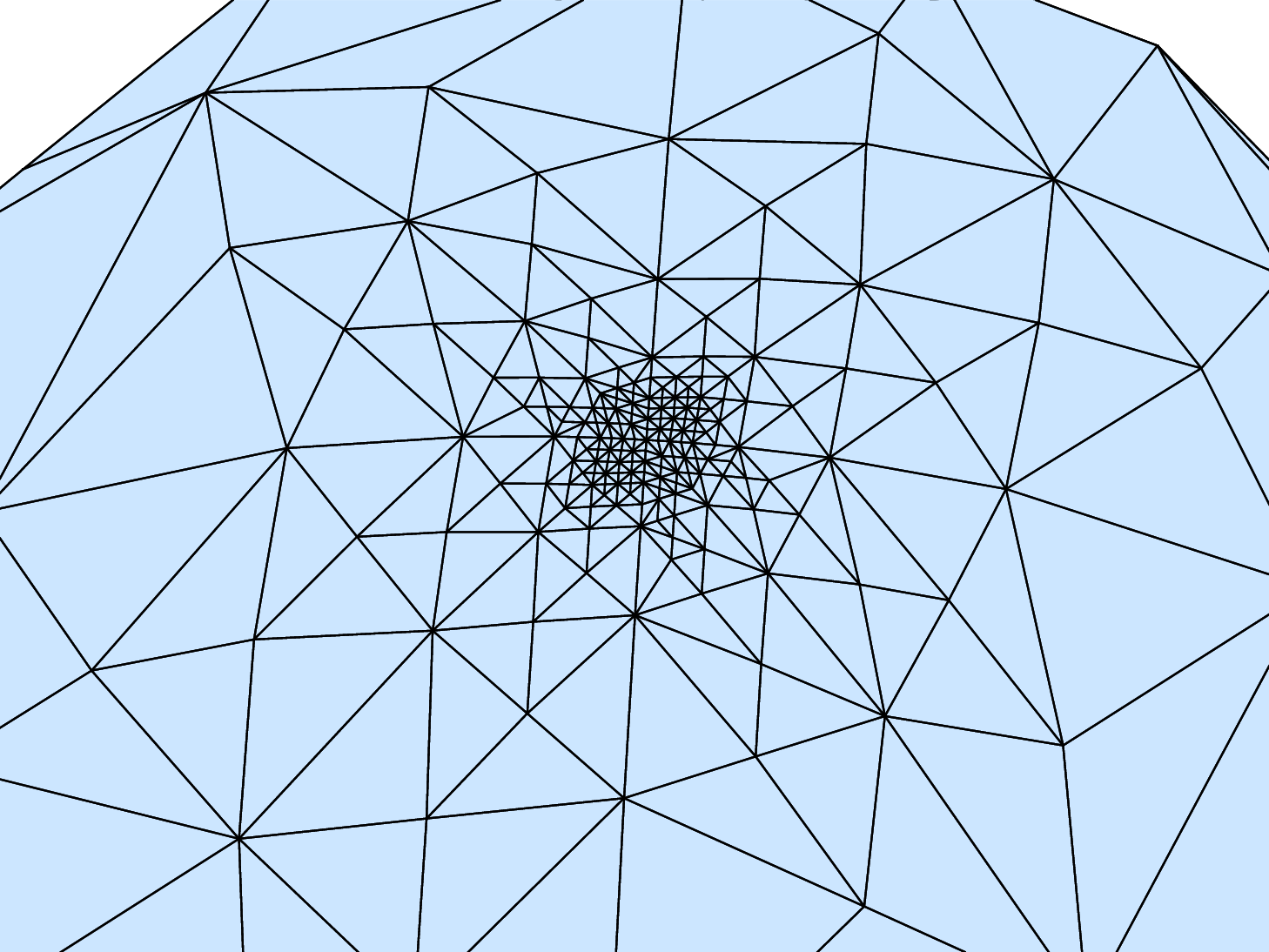}
  
  \vspace{0.2cm}
  
  \caption{Residual adaptive mesh with 211 vertices}
  \label{fig:sub1}
\end{subfigure}%
\begin{subfigure}{.5\textwidth}
  \centering
  \includegraphics[width=7cm]{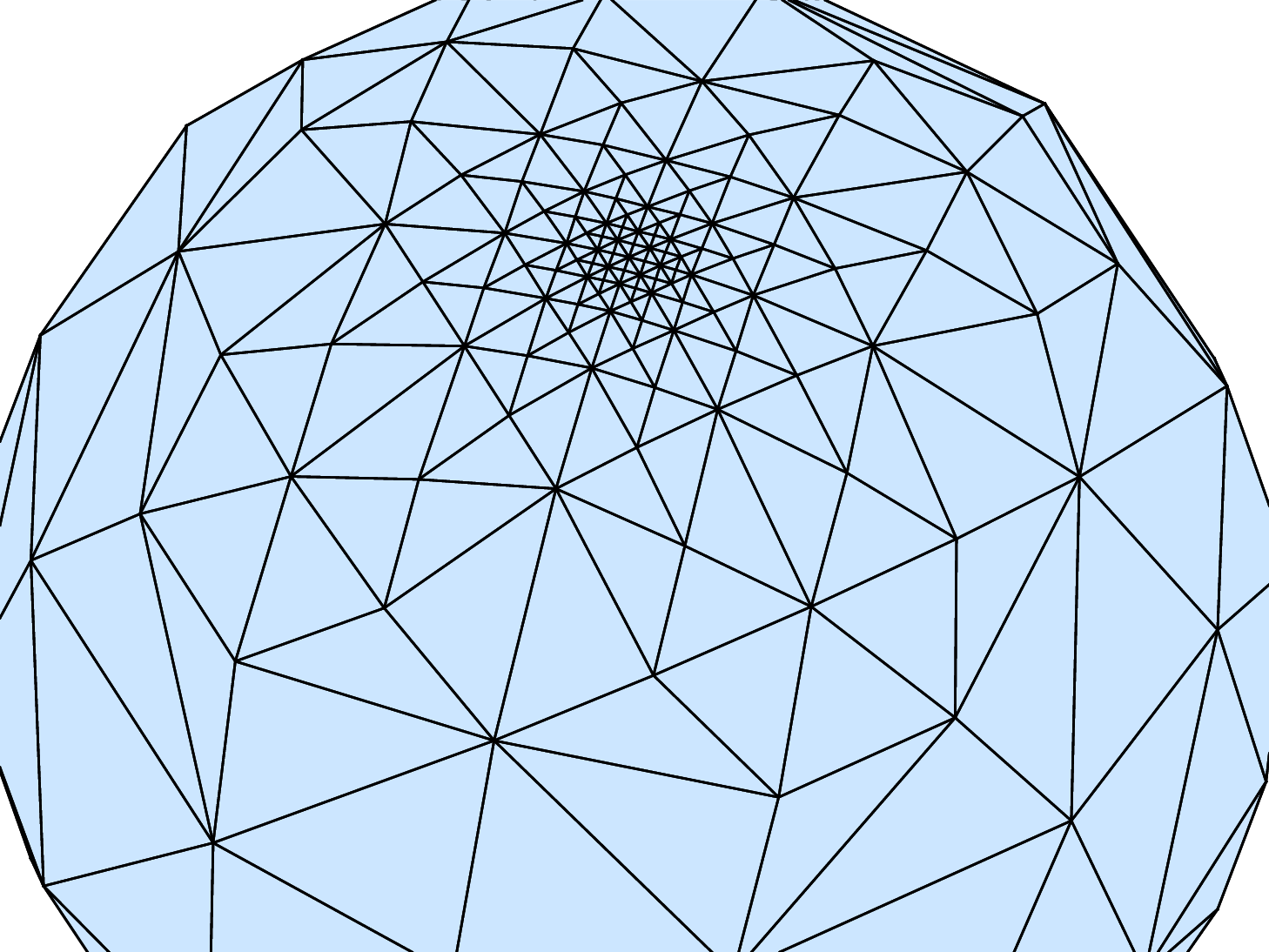}

  \vspace{0.2cm}
  
  \caption{hierarchical adaptive mesh with 170 vertices}
  \label{fig:sub2}
\end{subfigure}
\end{figure}

\begin{figure}[h]
\caption{Residual adaptive triangulation with $448$ vertices for $f_2$}
\label{f2:Re_adaptive448}

\vspace{0.5cm}

\centering
\begin{subfigure}{.5\textwidth}
  \centering
  \includegraphics[width=8cm]{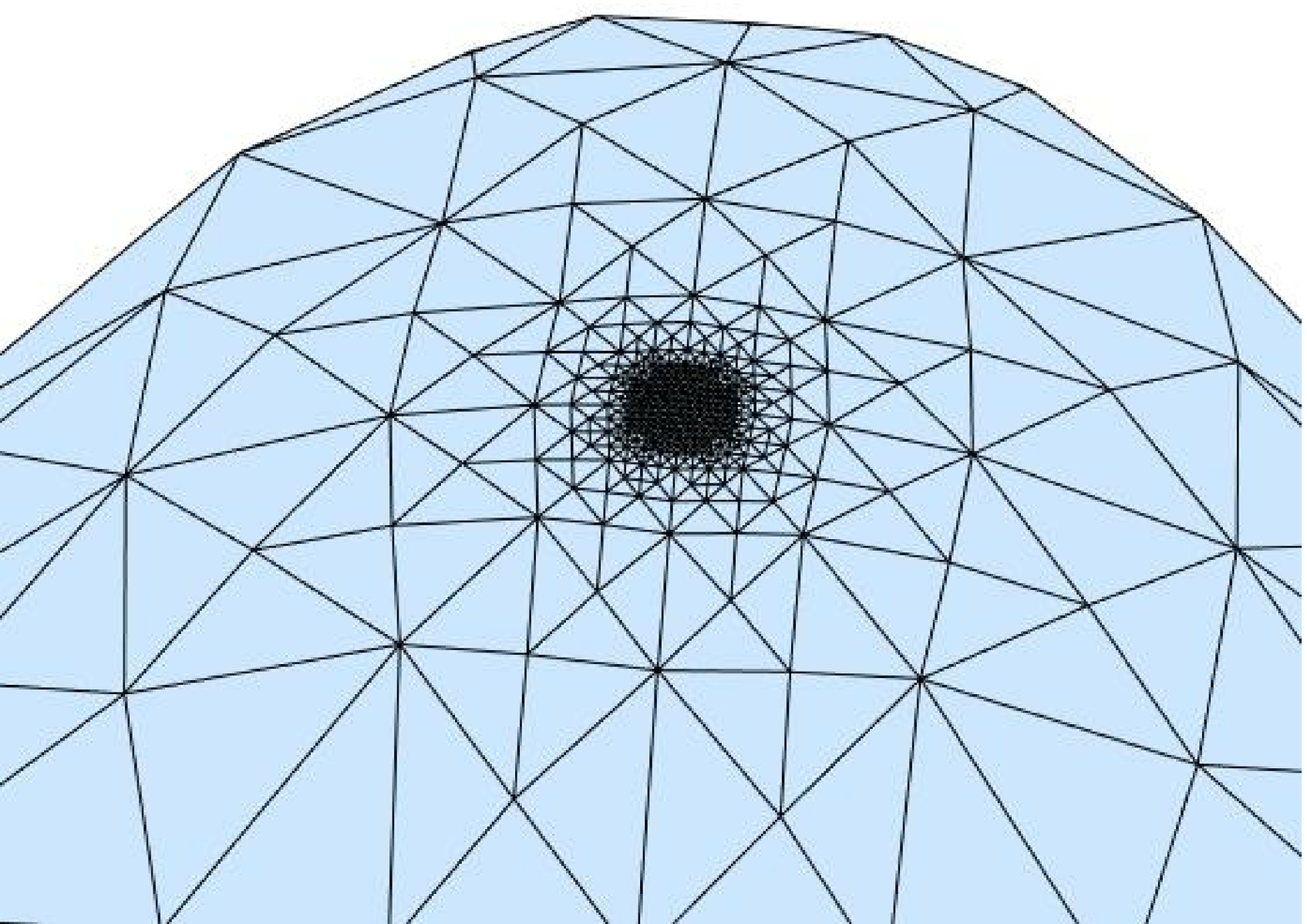}
  
  \vspace{0.2cm}
  
  \caption{At the North Pole}
  \label{fig:sub1}
\end{subfigure}%
\begin{subfigure}{.5\textwidth}
  \centering
  \includegraphics[width=8cm]{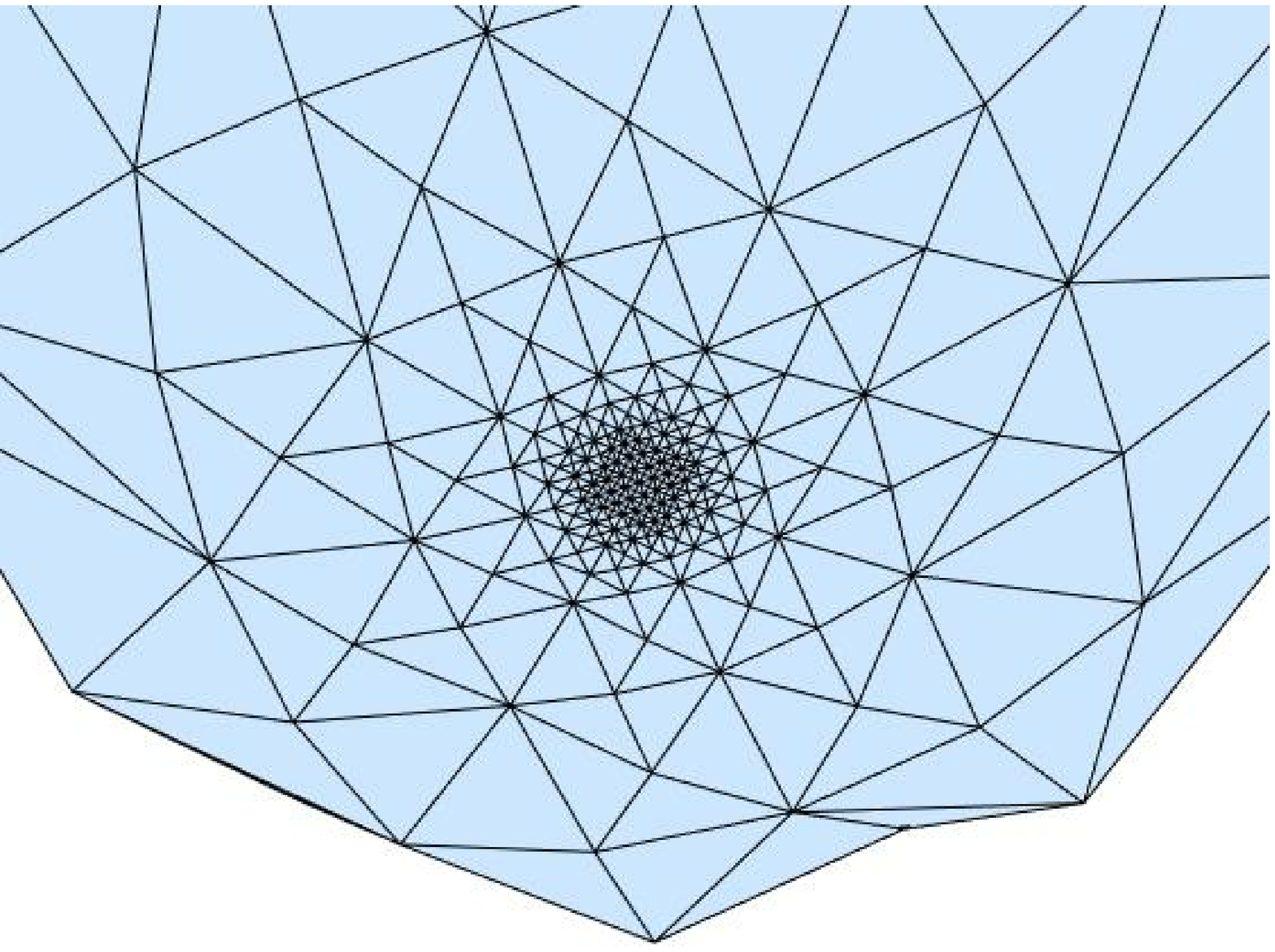}

  \vspace{0.2cm}
  
  \caption{At the South Pole}
  \label{fig:sub2}
\end{subfigure}
\end{figure}

\begin{figure}[h]
\begin{center}
\caption{hierarchical adaptive triangulation with $302$ vertices for $f_2$}
\label{f2:ReAdap302}

\vspace{0.5cm}

\centering
\begin{subfigure}{.5\textwidth}
  \centering
  \includegraphics[width=8cm]{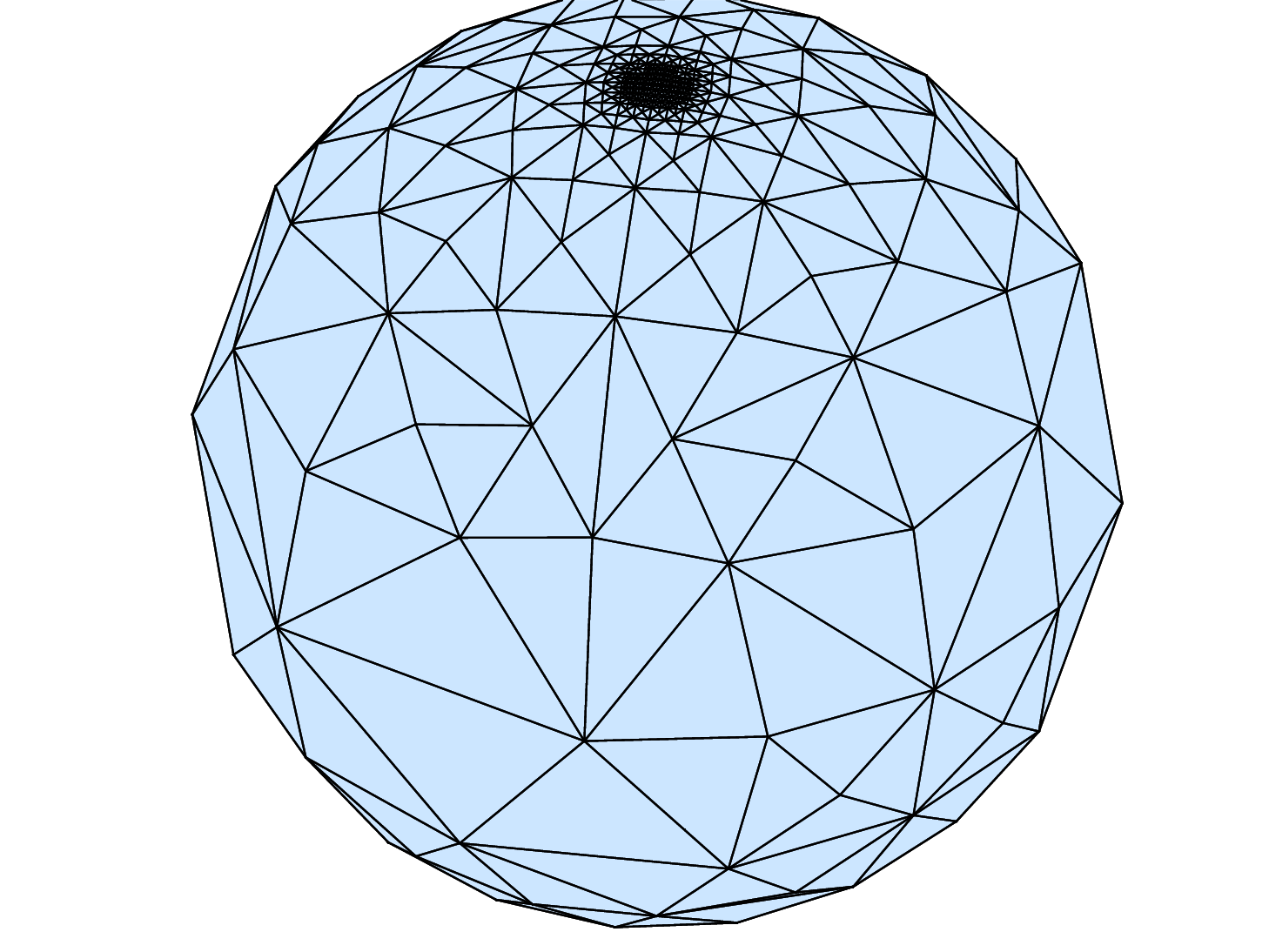}
  
  \vspace{0.2cm}
  
  \caption{At the North Pole}
  \label{fig:sub1}
\end{subfigure}%
\begin{subfigure}{.5\textwidth}
  \centering
  \includegraphics[width=8cm]{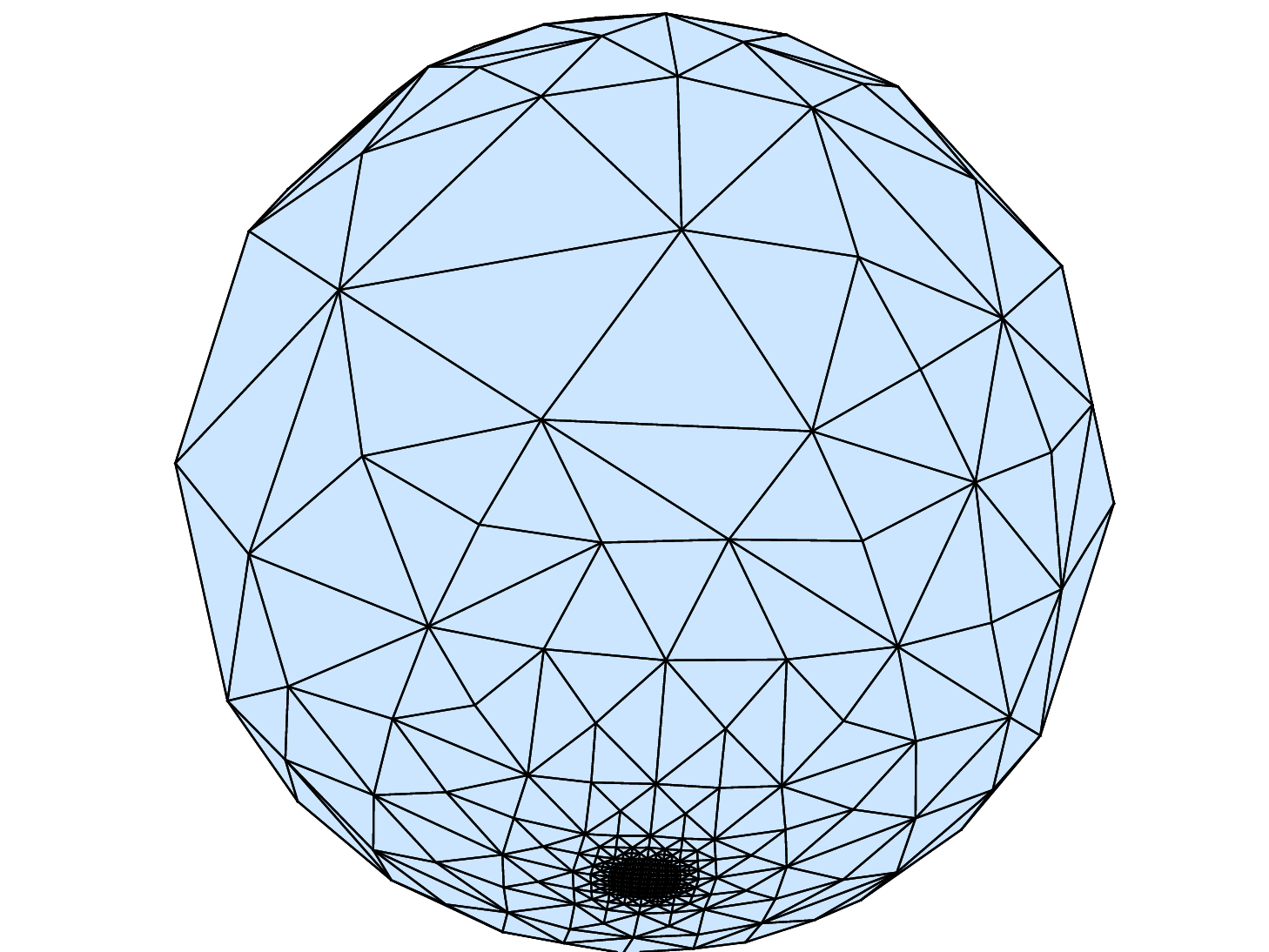}

  \vspace{0.2cm}
  
  \caption{At the South Pole}
  \label{fig:sub2}
\end{subfigure}
\end{center}
\end{figure}
% \end{comment}

\section*{Acknowledgement}
This research is funded by Vietnam National Foundation for Science and 
Technology Development (NAFOSTED) under grant number 101.99--2016.13.

\end{document}